%% file: master.tex
\documentclass[11pt]{amsart}
\usepackage{calc,amssymb,amsmath,ulem}
\usepackage{alltt}
\RequirePackage[dvipsnames,usenames]{color}

\input{kmacros3.sty}
\input{xy}
\input{chancery.sty}
\xyoption{all}
\usepackage{hyperref}
\usepackage{graphicx}
\usepackage[all,cmtip]{xy}
\usepackage{verbatim,xspace}
\usepackage[left=1.25in,top=1.09in,right=1.25in,bottom=1.09in]{geometry}
\usepackage{mabliautoref}

\newcommand{\khat}[1]{{{#1}'}}

\renewcommand{\O}{\cO}
\newcommand{\QQ}{\bQ}
\newcommand{\RR}{\bR}
\newcommand{\cG}{\sG}
\newcommand{\Rscr}{\sR}
\newcommand{\CM}{Castelnuovo-Mumford}
\newcommand{\fm}{\mathfrak{m}}
\newcommand{\cK}{\mathcal{K}}

\newcommand{\fa}{\fra}

\newcommand{\Xhat}{{X'}}
\newcommand{\Deltahat}{{\Delta'}}

\renewcommand{\ord}{\operatorname{ord}}

\begin{document}
\numberwithin{equation}{theorem}
\title[Test ideals and anti-canonical algebras]{Test ideals in rings with finitely generated anti-canonical algebras}
\author[Chiecchio]{Alberto Chiecchio}
\address{TASIS in Dorado}
\email{alberto.chiecchio@gmail.com}
\author[Enescu]{Florian Enescu}
\address{Department of Mathematics and Statistics,
Georgia State University,
Atlanta GA 30303 USA }
\email{fenescu@gsu.edu}
\author[Miller]{Lance Edward Miller}
\address{Department of Mathematical Sciences, University of Arkansas, Fayetteville, AR 72701}
\email{lem016@uark.edu}
\author[Schwede]{Karl Schwede}
\address{Department of Mathematics, University of Utah, 155 S 1400 E Room 233, Salt Lake City, UT, 84112}
\email{schwede@math.utah.edu}

\thanks{The fourth named author was supported in part by the
  NSF FRG Grant DMS \#1265261/1501115, NSF CAREER Grant DMS \#1252860/1501102 and a Sloan
  Fellowship.}



\begin{abstract}
Many results are known about test ideals and $F$-singularities for $\bQ$-Gorenstein rings.  In this paper we generalize many of these results to the case when the symbolic Rees algebra $\O_X \oplus \O_X(-K_X) \oplus \O_X(-2K_X) \oplus \ldots$ is finitely generated (or more generally, in the log setting for $-K_X - \Delta$).  In particular, we show that the $F$-jumping numbers of $\tau(X, \ba^t)$ are discrete and rational.  We show that test ideals $\tau(X)$ can be described by alterations as in Blickle-Schwede-Tucker (and hence show that splinters are strongly $F$-regular in this setting -- recovering a result of Singh).  We demonstrate that multiplier ideals reduce to test ideals under reduction modulo $p$ when the symbolic Rees algebra is finitely generated.  We prove that Hartshorne-Speiser-Lyubeznik-Gabber type stabilization still holds.  We also show that test ideals satisfy global generation properties in this setting.
\end{abstract}

\maketitle

\section{Introduction}

Test ideals were introduced by Hochster and Huneke in their theory of tight closure \cite{HochsterHunekeTC1} within positive characteristic commutative algebra.  After it was discovered that test ideals were closely related to multiplier ideals \cite{SmithMultiplierTestIdeals,HaraInterpretation}, a theory of test ideals of pairs was developed analogous to the theory of multiplier ideals \cite{HaraYoshidaGeneralizationOfTightClosure,TakagiInterpretationOfMultiplierIdeals}.  However, unlike multiplier ideals, test ideals were initially defined even without the hypothesis that $K_X$ was $\bQ$-Cartier (see \cite{DeFernexHaconSingsOnNormal} for a similar theory of multiplier ideals).  But the $K_X$ $\bQ$-Cartier hypothesis is useful for test ideals, and indeed a number of central open questions are still unknown without it.  The goal of this paper is generalize results from the hypothesis that $K_X$ is $\bQ$-Cartier to the setting where the local section ring $\sR(-K_X) := \O_X \oplus \O_X(-K_X) \oplus \O_X(-2K_X) \oplus \dots$ (also known as the symbolic Rees algebra) is finitely generated.

Most notably, perhaps the most important open problem within tight closure theory is the question whether weak and strong $F$-regularity are equivalent or more generally, whether splinters and strong $F$-regularity are equivalent (from the characteristic zero perspective, splinters, weak and strong $F$-regularity are competing notions of singularities analogous to KLT singularities that are all known to coincide in the $\bQ$-Gorenstein case).  These are known to be equivalent under the $K_X$ $\bQ$-Cartier hypothesis and under some other conditions \cite{SinghQGorensteinSplinters,LyubeznikSmithCommutationOfTestIdealWithLocalization,LyubeznikSmithStrongWeakFregularityEquivalentforGraded,AberbachMacCrimmonSomeResultsOnTestElements}.  Previously A.~K.~Singh announced a proof that splinters with $\sR(-K_R)$ finitely generated are strongly $F$-regular \cite{SinghSplintersFRegFG}.  We recover a new proof of this result and in fact show something stronger.  We prove that the (big) test ideal is equal to the image of a multiplier-ideal-like construction involving alterations.

\begin{theoremA*}[\autoref{thm.MainResultOnAlterations}, \autoref{cor.TauFiniteAlterations}]
Suppose that $X$ is a normal $F$-finite integral scheme and that $\Delta$ on $X$ is an effective $\bQ$-divisor such that $S= \sR(-K_X-\Delta)$ is finitely generated.  Then there exists an alteration $\pi : Y \to X$ from a normal $Y$, factoring through $X' = \sheafproj S$ so that
\[
\tau(X, \Delta) = \Image\big( \pi_* \O_Y(\lceil K_Y - \pi^*(K_X + \Delta) \rceil) \to \O_X\big).
\]
\begin{itemize}
\item{} If $X$ is of finite type over a perfect field, one may take $Y$ to be regular by \cite{deJongAlterations}.
\item{} Alternately, one may take $\pi$ to be a finite map (in which case $Y$ is almost certainly not regular).
\end{itemize}

As a consequence we obtain that
\[
\tau(X, \Delta) = \bigcap_{\pi : Y \to X} \Image\big( \pi_* \O_Y(\lceil K_Y - \pi^*(K_X + \Delta) \rceil) \to \O_X\big)
\]
where $\pi$ runs over all alterations with $Y \to X$ factoring through $X'$.  Alternately, one can run $\pi$ over all finite maps.  If additionally $X$ is of finite type over a perfect field, then one can run $\pi$ over all regular alterations factoring through $X'$.
\end{theoremA*}
Actually we prove a stronger theorem allowing for triples $(X, \Delta, \fra^t)$ but we don't include it here to keep the statement simpler.
Note that in characteristic zero, the same intersection over regular alterations characterized multiplier ideals by \cite{DeFernexHaconSingsOnNormal} at least after observing \autoref{rem.DiscAndRatAndAlterationsForMultiplierIdeals}.

Inspired by the analog with multiplier ideals \cite{EinLazSmithVarJumpingCoeffs}, there has been a lot of interest in showing that the jumping numbers of test ideals are rational and without limit points \cite{BlickleMustataSmithDiscretenessAndRationalityOfFThresholds,BlickleMustataSmithFThresholdsOfHypersurfaces,KatzmanLyubeznikZhangOnDiscretenessAndRationality,BlickleSchwedeTakagiZhang,BlickleTestIdealsViaAlgebras,SchwedeTuckerNonPrincipalIdeals,KatzmanSchwedeSinghZhang}. At this point, we know that the $F$-jumping numbers are discrete and rational for any $F$-finite scheme $X$ with $K_X$ $\bQ$-Cartier.  We also know discreteness if $\sR(-K_X)$ is finitely generated and $X = \Spec R$ is the spectrum of a graded ring \cite{KatzmanSchwedeSinghZhang}.  On the other hand we know that the jumping numbers of $\mJ(X, \Delta, \ba^t)$ are discrete and rational if $\sR(-K_X-\Delta)$ is finitely generated (see \autoref{rem.DiscAndRatAndAlterationsForMultiplierIdeals}).  We prove the following:
\begin{theoremB*}[\autoref{thm.DiscAndRatOfFjumping}, \autoref{prop.DiscretenessOfFJumpingNumbersViaGlobal}]
Suppose that $(R, \Delta)$ is a pair such that $\sR(-K_R - \Delta)$ is finitely generated.  Then for any ideal $\ba \subseteq R$ the $F$-jumping numbers of $\tau(R, \Delta, \ba^t)$ are rational and without limit points.
\end{theoremB*}
We prove the discreteness result in two ways.  First pass to the local section ring / symbolic Rees algebra, where the pullback of $-K_X-\Delta$ is $\bQ$-Cartier.  We then show that the test ideal of the symbolic Rees algebra restricts to the test ideal of the original scheme.  Alternately, in \autoref{sec.StablizationViaPositivity}, we prove the discreteness result for projective varieties by utilizing the theory developed by the first author and Urbinati \cite{ChiecchioUrbinatiAmpleWeilDivisors}.  In particular, we show a global generation result for test ideals, \autoref{thm.GGofNonPrincipalMult}, which immediately implies the test ideal result.

Another setting where the $\bQ$-Gorenstein hypothesis is used is in the study of non-$F$-pure ideals.  Recall that if $R$ is $F$-finite, normal and $\bQ$-Gorenstein with index not divisible by $p$, then it follows from \cite{HartshorneSpeiserLocalCohomologyInCharacteristicP,LyubeznikFModulesApplicationsToLocalCohomology,Gabber.tStruc} that the images of the evaluation-at-1 map
\[
\Hom_R(F^e_* R, R) \to R
\]
stabilize for sufficiently large $e$.  This stable image gives a canonical scheme structure to the non-$F$-pure locus of a variety.  We generalize this to the case that $\sR(-K_R)$ is finitely generated (which includes the case where the index of $-K_R$ is divisible by $p$).
\begin{theoremC*}[\autoref{cor.HartshorneSpeiserStabilization}, \autoref{thm.HartshorneSpeiserStabilization}, \autoref{thm.GlobalizedSigmaStabilization}]
Suppose that $R$ is an $F$-finite normal domain and that $B \geq 0$ is a $\bQ$-divisor with Weil-index not divisible by $p$.  If the anti-log-canonical algebra $\sR(-K_R-B)$ is finitely generated, then the image of the evaluation at 1 map $\Hom_R(F^e_* R( (p^e - 1) B), R) \to R$ stabilizes for $e$ sufficiently divisible.
\end{theoremC*}
Again we give several different proofs of this fact utilizing different strategies as above.  Finally, we also show:
\begin{theoremD*}
[\autoref{thm.Reduction}]
Suppose that $X$ is a normal quasi-projective variety over an algebraically closed field of characteristic zero.  Further suppose that $\Delta \geq 0$ is a $\bQ$-divisor such that $\sR(-K_X-\Delta)$ is finitely generated and also suppose that $\fra \subseteq \O_X$ is an ideal and $t \geq 0$ is a rational number.   Then
\[
\mJ(X, \Delta, \fra^t)_p = \tau(X_p, \Delta_p, \fra_p^t)
\]
for $p \gg 0$.
\end{theoremD*}
This should be compared with \cite{deFernexDocampoTakagiTuckerNumQGor} where the analogous result is shown under the hypothesis that $K_X + \Delta$ is numerically $\bQ$-Cartier.  This numerically $\bQ$-Cartier condition is somewhat orthogonal to the finite generation of $\sR(-K_X-\Delta)$.  In particular, if $\sR(-K_X-\Delta)$ is finitely generated and $-K_X-\Delta$ is numerically $\bQ$-Cartier then it is not difficult to see that $-K_X-\Delta$ is $\bQ$-Cartier, see also \cite{BoucksomdeFernexFavreUrbinati}.

\begin{remark}
A previous version of this paper included an incorrect statement in \autoref{lem.PullBackDivisorsOnSymbolicRees}.  This version, which also corrects the published version, fixes the statement by making it weaker.  Fortunately, we only needed the weaker statement in all our applications.
\end{remark}

\vskip 12pt
{\it Acknowledgements:}  The authors would like to thank Tommaso de Fernex, Christopher Hacon, Nobuo Hara, Mircea \mustata{} and Anurag K. Singh for several useful discussions.  We would also like to thank (Juan) Felipe P\'erez for several useful comments on a previous draft of this paper.  Finally, we would like to thank the referee for numerous valuable comments and for pointing out a mistake in a lemma (previously Lemma 2.15, now removed).

\section{Preliminaries}
\label{sec.Preliminaries}

\input{Preliminaries.tex}

\section{Stabilization, discreteness and rationality via Rees algebras}
\label{sec.StabilizationViaRees}

\input{ReesAlgebrasStabilization.tex}

\section{Stabilization and discreteness via positivity}
\label{sec.StablizationViaPositivity}

\input{PositivityStabilization.tex}



\section{Alterations}
\label{sec.Alterations}

\input{Alterations.tex}

\section{Reduction from characteristic zero}
\label{sec.Reduction}
\input{Reduction.tex}

\bibliographystyle{skalpha}
\bibliography{MainBib}

\end{document}

%% file: Preliminaries.tex
In this section we recall the basic properties that we will need of test ideals, local section rings\footnote{Also called (divisorial) symbolic Rees algebras in the commutative algebra literature.}, as well as the theory of positivity for non-$\bQ$-Cartier divisors as developed by Urbinati-Chiecchio \cite{ChiecchioUrbinatiAmpleWeilDivisors}.  We conclude by stating a finite generation result for local section rings of threefolds in positive characteristic (as a consequence of recent breakthroughs in the MMP).

\begin{setting}
Throughout this paper, all rings will be assumed to be Noetherian, of equal characteristic $p > 0$, and $F$-finite (which implies that they are excellent and have dualizing complexes \cite{KunzOnNoetherianRingsOfCharP,Gabber.tStruc}).  All schemes will be assumed to be Noetherian, $F$-finite, separated and have dualizing complexes.  For us a variety is a separated integral scheme of finite type over an $F$-finite field.  For any scheme $X$, we use $F : X \to X$ to denote the absolute Frobenius morphism.  We also make the following universal assumption:
\begin{equation}
\tag{$\dagger$}
F^! \omega_X^{\mydot} \cong \omega_X^{\mydot}.
\end{equation}
This holds for all schemes of essentially finite type over an $F$-finite field (or even of essentially of finite type over an $F$-finite local ring).

Frequently we will also consider divisors on schemes $X$.  Whenever we talk about divisors $\Delta$ on $X$, we make the universal assumption that $X$ is normal and integral.  In particular, whenever we consider a pair $(R, \Delta)$ or $(X, \Delta)$, then $R$ or $X$ is implicitly assumed to be normal.
\end{setting}

We make one remark on some nonstandard notation that we use.  If $R$ is a normal domain and $D$ is a Weil divisor on $X = \Spec R$, then we use $R(D)$ to denote the fractional ideal $H^0(X, \cO_X(D)) \subseteq K(R)$.

\subsection{Test ideals and $F$-singularities}

We now recall the definitions and basic properties of test ideals.  While test ideals were introduced in \cite{HochsterHunekeTC1}, we are technically talking about the \emph{big/non-finitistic test ideal} from \cite{LyubeznikSmithCommutationOfTestIdealWithLocalization,HochsterFoundations}.  The particular definition of the test ideal presented here can be found in \cite[Definition 9.3.8]{BlickleSchwedeSurveyPMinusE} among other places.

\begin{definition}[Test ideals]
\label{defn.TestIdeals}
Suppose that $R$ is an $F$-finite normal domain, $\Delta \geq 0$ is a $\bQ$-divisor, $\ba \subseteq R$ is a non-zero ideal sheaf and $t \geq 0$ is a real number.  The test ideal
\[
\tau(X, \Delta, \ba^t)
\]
is the unique smallest nonzero ideal $J \subseteq R$ such that for every $e > 0$ and every $$\phi \in (F^e_* \ba^{\lceil p^e t \rceil}) \cdot \Hom_R(F^e_* R(\lceil (p^e-1) \Delta \rceil), R) \subseteq \Hom_R(F^e_*R, R)$$ we have that $\phi(F^e_* J) \subseteq J$.

If $\Delta = 0$ then we leave it out writing $\tau(X, \ba^t)$.  If $\ba = R$ or $t = 0$ then we write $\tau(R, \Delta)$.
\end{definition}

It is not obvious that the test ideal exists.  However, it can be shown that there exists $c \in R$ such that for each $0 \neq d \in R$, we have that $c \in \sum_{e > 0} \sum_{\phi} \phi(F^e_* (d R))$ where $\phi$ varies over $(F^e_* \ba^{\lceil t(p^e - 1) \rceil}) \cdot \Hom_R(F^e_* R(\lceil (p^e-1) \Delta \rceil), R)$, see \cite[Lemma 3.21]{SchwedeTestIdealsInNonQGor}.  This element $c$ is then called a \emph{(big) $(R, \Delta, \ba^t)$-test element}.  We then immediately obtain the following construction of the test ideal.

\begin{lemma}
\label{lem.TestIdealConstruction}
With notation as in \autoref{defn.TestIdeals}, if $c$ is a big $(R, \Delta, \ba^t)$-test element, then
\[
\tau(R, \Delta, \ba^t) = \sum_{e \geq 0} \sum_{\phi} \phi(F^e_* (c R))
\]
where again $\phi$ ranges over elements of $(F^e_* \ba^{\lceil t(p^e - 1) \rceil}) \cdot \Hom_R(F^e_* R(\lceil (p^e-1) \Delta \rceil), R)$.

One may also range $\phi$ over elements of $(F^e_* \ba^{\lceil tp^e \rceil}) \cdot \Hom_R(F^e_* R(\lceil (p^e-1) \Delta \rceil), R)$.  Alternately, one may replace $e \geq 0$ with $e \gg 0$.

Finally, we also have that for any sufficiently large Cartier divisor $D \geq 0$ that
\[
\tau(R, \Delta, \ba^t) = \sum_{e \geq e_0} \Tr^e F^e_* \big(\fra^{\lceil t p^e \rceil} \cdot \O_X( K_X - p^e(K_X + \Delta) - D) \big).
\]
\end{lemma}
\begin{proof}  For the first statement, it is easy to see that $c$ is contained in any ideal satisfying the condition $\phi(F^e_* J) \subseteq J$ for all $\phi \in (F^e_* \ba^{\lceil t(p^e - 1) \rceil}) \cdot \Hom_R(F^e_* R(\lceil (p^e-1) \Delta \rceil), R)$.  Hence so is the sum.  Thus the sum is the smallest such ideal.

For the second statement replacing $\ba^{\lceil t(p^e - 1) \rceil}$ with $\ba^{\lceil tp^e \rceil}$ obviously we have the $\tau \supseteq \sum$ containment.  Notice that if $c$ is a test element, then so is $dc$ for any $0 \neq d \in R$.  Hence one can form the original sum with $cd$ for some $d$ so that $d \ba^{\lceil t(p^e-1) \rceil} \subseteq \ba^{\lceil t p^e \rceil}$ for all $e$.  The inclusion $\tau \subseteq \sum$ follows.

For the $e \gg 0$ statement, notice that if $J$ is the sum for $e \gg 0$, then we still have $\phi(F^e_* J) \subseteq J$.  The final characterization of the test ideal follows immediately from the fact that
\[
\fra^{\lceil t p^e \rceil} \cdot \O_X( K_X - p^e(K_X + \Delta) - \Div_X(c))  = \fra^{\lceil t p^e \rceil} \cdot \O_X(-D) \cdot \sHom_{\O_X}(F^e_* \O_X( \lfloor p^e \Delta \rfloor), \O_X ).
\]
We notice that any difference coming from the fact that we round down instead of round up can be absorbed into the difference between $D$ and $\Div_X(c)$.
\end{proof}

We also recall some properties of the test ideal for later use.

\begin{lemma}
\label{lem.BasicTestIdealProperties}
Suppose that $(R, \Delta, \ba^t)$ is as in \autoref{defn.TestIdeals}.  Then:
\begin{enumerate}
\item The formation of $\tau(R, \Delta, \ba^t)$ commutes with localization and so one can define $\tau(X, \Delta, \ba^t)$ for schemes as well.
\item If $s \geq t$, then $\tau(R, \Delta, \ba^s) \subseteq \tau(R, \Delta, \ba^t)$.
\item For any $t \geq 0$, there exists an $\varepsilon > 0$ so that if $s \in [t, t+\varepsilon)$, then $\tau(X, \Delta, \ba^t) = \tau(X, \Delta, \ba^s)$.
\item If $0 \neq f \in R$ and $H = V(f)$ is the corresponding Cartier divisor, then
\[
f\tau(X, \Delta, \ba^t) = \tau(X, \Delta, \ba^t) \otimes \O_X(-H) = \tau(X, \Delta+H, \ba^t).
\]
\end{enumerate}
\end{lemma}
\begin{proof}
Part (a) follows immediately from \autoref{lem.TestIdealConstruction}.  Part (d) follows similarly (use the projection formula).  Part (b) is obvious also from \autoref{lem.TestIdealConstruction}.

For part (c), this is \cite[Exercise 9.12]{BlickleSchwedeSurveyPMinusE}.  Let us quickly sketch the proof since we do not know of a reference where this is addressed in full generality.  Choose $c \neq 0$ a test element.  It is easy to choose $c$ that works for all all $s \in [t, t+1]$.  We then write $\tau(R, \Delta, \ba^t) = \sum_{e \geq 0} \sum_{\phi} \phi(F^e_* (cd R))$ for $d$ some element in $\ba$.  This sum is a finite sum, say for $e = 0$ to $e = m$.  Let $\varepsilon = {1 \over p^m}$.  Then
\[
\tau(R, \Delta, \ba^t) = \sum_{e=0}^m \sum_{\phi} \phi(F^e_* (cd R)) = \sum_{e = 0}^m \sum_{\psi} \psi(F^e_* (cd\ba^{\lceil tp^e \rceil})) \subseteq \sum_{e = 0}^m \sum_{\psi} \psi(F^e_* (c\ba^{\lceil tp^e \rceil + 1}))
\]
where $\psi$ now runs over $\Hom_R(F^e_* R(\lceil (p^e - 1)\Delta \rceil), R)$.

 We then see that $\lceil t(p^e-1) \rceil + 1 \geq \lceil (t + \varepsilon)p^e \rceil$ for $e \leq m$ and so
\[
\tau(R, \Delta, \ba^t) \subseteq \sum_{e = 0}^m \sum_{\psi} \psi(F^e_* (c\ba^{\lceil (t + \varepsilon)p^e \rceil})) \subseteq \tau(R, \Delta, \ba^{t + \varepsilon}).
\]
The other containment was handled in (b).
\end{proof}

Finally, we make one more definition related to test ideals.

\begin{definition}
A triple $(X, \Delta, \ba^t)$ as in \autoref{defn.TestIdeals} is called \emph{strongly $F$-regular} if $\tau(R, \Delta, \ba^t) = R$.
\end{definition}

We briefly also recall some formalities of $p^{-e}$-linear maps and connections with divisors.
\begin{lemma}
\label{lem.MapsAndDivisorsBasics}
Suppose that $X = \Spec R$ is an $F$-finite normal scheme.
\begin{enumerate}
\renewcommand{\theenumi}{\alph{enumi}}
\item There is a bijection between effective divisors $\Delta$ such that $(p^e -1)(K_R + \Delta) \sim 0$ and elements $\phi$ of $\Hom_R(F^e_* R, R)$ modulo pre-multiplication by units.  We use the following notation for this correspondence:
\[
\begin{array}{rcl}
\phi & \mapsto & \Delta_{\phi},\\
\Delta & \mapsto & \phi_{\Delta}.
\end{array}
\]
\item If $\phi \in \Hom_R(F^e_* R, R)$ corresponds to a $\bQ$-divisor $\Delta_{\phi}$, then the map $$\psi(\blank) = \phi(F^e_* (d \cdot \blank))$$ corresponds to the divisor $\Delta_{\psi} = \Delta_{\phi} + {1 \over p^{e}-1} \divisor(d)$.
\item If $\phi \in \Hom_R(F^e_* R, R)$ and $\phi^n := \phi \circ (F^e_* \phi) \circ \ldots \circ (F^{(n-1)e}_*) : F^{ne}_* R \to R$, then $\Delta_{\phi^n} = \Delta_{\phi}$.
\item Using the bijection of (a), if $\Delta \geq 0$ is any effective $\bR$-divisor, then the elements $\phi$ of $\Hom_R(F^e_* R(\lceil (p^e-1) \Delta \rceil) \subseteq \Hom_R(F^e_* R, R)$ (modulo multiplication by units) are in bijection with divisors $\Delta_{\phi}$ with $\Delta_{\phi} \geq \Delta$ and of course with $(p^e-1)(K_R + \Delta) \sim 0$.
\end{enumerate}
\end{lemma}
\begin{proof}
(a) is just \cite[Theorem 3.13]{SchwedeFAdjunction}. (b) is a straightforward exercise, see \cite[Exercise 4.11]{BlickleSchwedeSurveyPMinusE}. (c) is \cite[Theorem 3.11(e)]{SchwedeFAdjunction}.  (d) is also not difficult to check, see for instance \cite[Definition 6.8]{SchwedeTuckerTestIdealFiniteMaps}.
\end{proof}

\subsection{Local section rings of divisors / symbolic Rees algebras}
\label{subsec.LocalSectionRings}

Suppose that $X$ is an $F$-finite normal integral Noetherian scheme and $\Gamma$ is a $\bQ$-divisor on $X$.  Then one can form
\[
S = \sR(X, \Gamma) = \bigoplus_{i \in \bZ_{\geq 0}} \O_X(\lfloor i \Gamma \rfloor).
\]
Additionally for any integer $n > 0$ we use $S^{(n)} := \sR(X, n\Gamma)$ to denote the $n$th Veronese subalgebra.  Note that there is a canonical map $\O_X \to S$ and dually a map $\sheafspec S \xrightarrow{\kappa} X$ of schemes (note $S$ may not be Noetherian).  If $S$ (or equivalently $Y = \sheafspec S$) is Noetherian, then we also have $\sheafproj S \xrightarrow{\mu} X$.  These maps are very well behaved outside of codimension $\geq 2$. We recall that the map $\sheafproj S \xrightarrow{\mu} X$ is called the \emph{$\QQ$-Cartierization} of $\Gamma$.  It is a small projective morphism $Y\xrightarrow{f}X$ such that the strict transform $f^{-1}_*\Gamma$ is $\QQ$-Cartier and $f$-ample. Moreover, such a map exists if and only if $S$ is Noetherian, \cite[Lem. 6.2]{KollarMori}. When $S$ is finitely generated, both $\sheafspec S$ and $\sheafproj S$ are normal, see for instance \cite{DemazureNormalGradedRings} and also see \cite{WatanabeRemarksOnDemazure}.

\begin{lemma}
Suppose that $S$ is finitely generated and $W \subseteq X$ is a closed subset of codimension $\geq 2$.  Then $\kappa^{-1} W$ and $\mu^{-1} W$ are also codimension $\geq 2$ in $\sheafspec S$ and $\sheafproj S$ respectively.  Additionally $\mu$ is an isomorphism outside a closed codimension $\geq 2$ subset of $X$ and if $\Gamma$ is integral, then $\kappa$ is an $\bA^1$-bundle outside a set of codimension 2.

As a consequence, if $D$ is any $\bQ$-divisor on $X$, then we have canonical pullbacks $\mu^* D$ and $\kappa^* D$.
\end{lemma}
\begin{proof}
Since $S$ is a symbolic Rees algebra (of a module of rank $1$), the map $\mu:\sheafproj S\rightarrow X$ is small, \cite[Lem. 6.2]{KollarMori}. The case for $\kappa$ can be verified locally on $X$.

We begin under the assumption that $\Gamma$ is integral.  Let $U=\Spec R\subseteq X$ and $S'=S\big|_U$.  Suppose that $\kappa^{-1} W$ has a codimension $1$ component whose support is defined by a prime height 1 ideal $Q$ in $S$.  By \cite[Lemma 4.3]{GotoHerrmannNishidaVillamayor}, $Q \cap R$ is height zero or 1.  But this is impossible since it defines a subset of $W$, a set of codimension 2.  For the case when $\Gamma$ is not integral, we observe that the result holds for the Veronese algebra $S^{(n)} = \sR(X, n\Gamma)$ for sufficiently divisible $n$.  But $S$ is a finite $S^{(n)}$ algebra (see the proof of \cite[Lemma 2.4]{GotoHerrmannNishidaVillamayor}) and the result follows.

 The fact that $\kappa$ is a $\bA^1$ bundle, at least outside a set of codimension 2, follows immediately from the fact that in that case, $\Gamma$ is an integral Cartier divisor and so the section ring $S$ looks locally like $\O_X[t]$ outside a set of codimension 2.
\end{proof}

\begin{rmk} When $X$ is separated, the pullback $\mu^*\Gamma$ coincides with the pullback of de~Fernex and Hacon, \cite{DeFernexHaconSingsOnNormal}; see \autoref{rmk.pullbacksmall}.
\end{rmk}

We will be very interested in proving that various section rings are finitely generated and so recall:
\begin{samepage}
\begin{lemma}
\label{lem.BasicPropertiesOfLocalSectionAlgebras}
With notation as defined at the start of \autoref{subsec.LocalSectionRings},
\nopagebreak
\begin{itemize}
\item[(a)]  $S$ is finitely generated if and only if $S^{(n)}$ is finitely generated for some (equivalently any) $n > 0$.

\item[(b)]  Suppose that $g : Y \to X$ is a finite dominant map from another normal integral Noetherian scheme $Y$.  Let $T = \sR(Y, g^*\Gamma)$.  Then $T$ is finitely generated if $S$ is finitely generated.
\end{itemize}
\end{lemma}
\end{samepage}
\begin{proof}
Part (a) is exactly \cite[Lemma 2.4]{GotoHerrmannNishidaVillamayor} (although it can also be found in numerous other sources).  For (b) we do not know a good reference but we sketch a proof here.  By (a), we may assume that $\Gamma$ is integral.  It is also harmless to assume that $X = \Spec A$ is affine and hence so is $Y = \Spec B$.  Then we can pass to the category of commutative rings so that $S$ and $T$ are actually rings (and not sheaves of rings).  In particular, we suppress all $g_*$ notation that we might otherwise need.  We have the diagram
\[
\xymatrix{
A \ar[r] \ar[d] & B \ar[d] \\
S \ar[r] & T \\
}
\]
First choose a single element $c \in A$ so that $S[c^{-1}] \cong A[c^{-1}, t]$ and $T[c^{-1}] \cong B[c^{-1}, t]$ (here we identify the $t$s).  It follows that $S[c^{-1}] \subseteq T[c^{-1}]$ is finite and hence integral and $K(S) \to K(T)$ is a finite extension as well.  Let $T'$ to be the integral closure of $S$ inside $T$.  We want to show that $T = T'$ which will complete the proof.

Recall we already assumed that $\Gamma$ is integral.  Let $W \subseteq X$ be a closed set of codimension 2 outside of which $\Gamma$ is Cartier.  Consider the functor $H^0(X \setminus W, \blank)$ applied to all of the rings (or sheaves of rings) involved.  As $T$ is a direct sum of reflexive $\O_Y$-modules, $H^0(X \setminus W, T) = H^0(Y \setminus g^{-1}(W), T)$ is just the global sections of $T$ by Hartog's Lemma for reflexive sheaves \cite{HartshorneGeneralizedDivisorsOnGorensteinSchemes}.  Thus $H^0(X \setminus W, T)$ is identified with $T$ since $X$ and $Y$ are affine.

On the other hand, $\kappa^{-1} W$ is a codimension $\geq 2$ subset of $\sheafspec S$, outside of which $T$ and $T'$ obviously agree.  Hence $H^0(X \setminus W, T) = H^0(X \setminus W, T') = H^0(\sheafspec S \setminus \kappa^{-1} W, T')$.  But $T'$ is normal and so we also have that $H^0(\sheafspec S \setminus \kappa^{-1} W, T') = T'$.  We have just shown that $T = T'$ as desired.
\end{proof}

We also will need to understand the canonical divisors of $\sheafspec S$ and $\sheafproj S$.

\begin{lemma}
\label{lem.PullBackDivisorsOnSymbolicRees}
Continuing with notation from the start of \autoref{subsec.LocalSectionRings}, assuming that $S$ is finitely generated, then $K_{\sheafproj S} \sim \mu^* K_X$.  If additionally, $\Gamma$ is a Weil divisor then, locally on the base, we have that $K_{\sheafspec S} \sim \kappa^* K_X + \kappa^* \Gamma$.  In particular, if $\Gamma = -K_X - B$, then $K_{\sheafspec S} \sim \kappa^*(-B)$ so that $K_{\sheafspec S} + \kappa^* B \sim 0$.  Thus if $\Gamma = -K_X$ then $S$ is quasi-Gorenstein.\footnote{Also called $1$-Gorenstein.}
\end{lemma}
\begin{proof}
Recall that $\kappa$ is an $\bA^1$-bundle outside a set of codimension 2 and hence $\kappa^*$ makes sense.
The computation of $K_{\sheafspec S}$ can be found in \cite[Theorem 4.5]{GotoHerrmannNishidaVillamayor}.  The initial statement that $K_{\sheafproj S} = \mu^* K_X$ is obvious since $\mu$ is small.
\end{proof}

\subsection{Positivity for non-$\bQ$-Cartier divisors}

In this section we will recall some definitions and results of \cite{ChiecchioUrbinatiAmpleWeilDivisors}. Let us recall that, if $f:Y\rightarrow X$ is a morphism of schemes, a coherent sheaf $\sF$ on $Y$ is \emph{relatively globally generated}, or \emph{$f$-globally generated}, if the natural map $f^*f_*\sF\rightarrow\sF$ is surjective. If $Y$ is a normal scheme and $D$ is a Weil divisor on $Y$, it might be that, for example, $\O_Y(D)$ is $f$-globally generated, but $\O_Y(2D)$ is not. To account for such pathologies we have to work asymptotically: we will say that a $\bQ$-divisor $D$ is \emph{relatively asymptotically globally generated}, or \emph{$f$-agg}, if $\O_Y(mD)$ is $f$-globally generated for all positive $m$ sufficiently divisible.

Let $f:Y\rightarrow X$ be a projective morphism of normal Noetherian schemes. A $\bQ$-Weil divisor $D$ on $X$ is \emph{$f$-nef} if, for every $f$-ample $\bQ$-Cartier $\bQ$-divisor $A$ on $Y$, $D+A$ is $f$-agg; if $X=\Spec k$, we will simply say that $D$ is \emph{nef}, \cite[Def. 2.4]{ChiecchioUrbinatiAmpleWeilDivisors}. The $\bQ$-divisor $D$ is \emph{$f$-ample} if, for every ample $\bQ$-Cartier $\bQ$-divisor $A$ on $Y$, there exists $b>0$ such that $bD-A$ is $f$-nef and the algebra of local sections $\sR(X,D)$ is finitely generated; if $X=\Spec k$, we will say that $D$ is \emph{ample}, \cite[Def. 2.14]{ChiecchioUrbinatiAmpleWeilDivisors}. Notice that when $D$ is $\Q$-Cartier, these notions coincide with the usual ones of nefness and amplitude. We remark that amplitude for Weil divisors is given by two conditions: a positivity one -- which is based on the fact that the regular ample cone is the interior of the nef cone -- and a technical one -- on the finite generation of the algebra of local sections. These two conditions are independent; in particular there are examples of Weil divisors $A$ satisfying the positivity condition, but with algebra of local sections $\sR(X,A)$ not finitely generated, \cite[Example 2.20]{ChiecchioUrbinatiAmpleWeilDivisors}.

These notions of positivity behave very much like in the $\QQ$-Cartier world: for example, if $A$ is an ample Weil $\QQ$-divisor, and $D$ is a globally generated $\QQ$-Cartier divisor, $D+A$ is ample, \cite[Lem. 2.18(i)]{ChiecchioUrbinatiAmpleWeilDivisors}.

\begin{lem} \label{lem:FiniteGenNonPrinc}
Let $E$ be a $\QQ$-divisor on a normal Noetherian projective scheme $X$ over a field $k$. If the algebra of local section $\sR(X,E)$ is finitely generated then there exists a Cartier divisor $L$ such that $L +E$ is an ample Weil divisor.
\end{lem}

\begin{proof} Notice that $n(L+E)$ is ample for some/every $n>0$, if and only if $L+E$ is ample, \cite[Lem. 2.18(b)]{ChiecchioUrbinatiAmpleWeilDivisors}. So, without loss of generality, we can assume that $\sR(X,E)$ is generated in degree $1$ and that $E$ is integral. Let $H$ be an ample Cartier divisor. By definition, there exists $m>0$ such that $\cO_X(mH+E)$ is globally generated. There is a surjection
$$
\cO_X(E)^{\otimes n}\twoheadrightarrow\cO_X(E)^{\cdot n}=\cO_X(nE),
$$
where the last equality is a consequence of the assumption on the finite generation of the algebra of local sections. Thus, for each $n>0$, $\cO_X(n(mH+E))$ is globally generated, that is, $mH+E$ is asymptotically globally generated. Moreover, since $H$ is Cartier, $\sR(X,mH+E)$ is also finitely generated. By \cite[Lem. 2.18(i)]{ChiecchioUrbinatiAmpleWeilDivisors}, $mH+E+H=(m+1)H+E$ is ample.
\end{proof}

The main characterization of the above positivity is in terms of their $\QQ$-Cartierization. Let $X$ is a normal projective Noetherian scheme over an algebraically closed field $k$, and let $D$ be a $\QQ$-divisor with $S=\sR(X,D)$ finitely generated. Let $\sheafproj S\xrightarrow{\mu} X$. Notice that $D'=\mu^{-1}_*D$ is $\QQ$-Cartier. Then $D$ is nef/ample if and only if so is $D'$, \cite[Theorems 3.3 and 3.6]{ChiecchioUrbinatiAmpleWeilDivisors}.

Using this characterization, Urbinati and the first author proved Fujita vanishing for locally free sheaves, \cite[Cor. 4.2]{ChiecchioUrbinatiAmpleWeilDivisors}. Let $X$ be a normal projective Noetherian scheme over an algebraically closed field $k$, let $A$ be an ample $\QQ$-divisor on $X$ and let $\sF$ be a locally free coherent sheaf on $X$. There exists an integer $m(A,\sF)$ such that $$H^i(X,\sF\otimes\cO_X(mA+D))=0$$ for all positive $m$ divisible by $m(A,\sF)$, all nef Cartier divisors $D$, and all $i>0$.



\subsection{Pullback of Weil divisors}\label{subsect.pullbackWeildivisors}

Let $f:Y\rightarrow X$ be a proper birational morphism of normal Noetherian separated schemes. In \cite{DeFernexHaconSingsOnNormal}, de Fernex and Hacon introduced a way of pulling back a Weil divisor on $X$ via $f$.

For any Weil divisor $D$ on $X$, the \emph{$\natural$-pullback} of $D$ along $f$, denoted by $f^\natural D$, is the Weil divisor on $Y$ such that
\begin{eqnarray}\label{eq:naturalpullbackchar}
\cO_Y(-f^\natural D)=(\cO_X(-D)\cdot\cO_Y){}^\vee{}^\vee,
\end{eqnarray}
\cite[Def 2.6]{DeFernexHaconSingsOnNormal}. The negative sign appearing is so that, when $D$ is effective, we are pulling back the ideal defining it as a subscheme. The \emph{pullback} of $D$ along $f$ is
$$
f^*D=\liminf_m\frac{f^\natural(mD)}{m}=\lim_m\frac{f^\natural(m!D)}{m!}\quad\textrm{coefficient-wise}.
$$

The above is well-defined, the infimum limit over $m$ is a limit over $m!$ and an $\RR$-divisor, \cite[Lem. 2.8 and Def. 2.9]{DeFernexHaconSingsOnNormal}. Moreover, the above definition of $f^*$ coincides with the usual one whenever $D$ is $\QQ$-Cartier, \cite[Prop. 2.10]{DeFernexHaconSingsOnNormal}.

\begin{rmk}\label{rmk.pullbacksmall} If $f:Y\rightarrow X$ is a small, projective birational morphism, then $f^*D=f^{-1}_*D$
\end{rmk}

This notion of pullback is not quite functorial, unfortunately. Let $f:Y\rightarrow X$ and $g:V\rightarrow Y$ be two birational morphisms of normal Noetherian separated schemes, and $D$ be a Weil divisor on $X$. The divisor $(fg)^\natural(D)-g^\natural f^\natural(D)$ is effective and $g$-exceptional. Moreover, if $\cO_X(-D)\cdot\cO_Y$ is an invertible sheaf, $(fg)^\natural(D)=g^\natural f^\natural(D)$, \cite[Lem. 2.7]{DeFernexHaconSingsOnNormal}.

\begin{lem}\label{lem.resofKX+D} Let $X$ be a normal Noetherian scheme. Let $D$ be a Weil divisor on $X$, such that $S=\sR(X,D)$ is finitely generated; let $\khat{X}=\sheafproj S\xrightarrow{\mu} X$ and let $\khat{D}:=\mu^{-1}_*D$. Then $\cO_X(mD)\cdot \cO_{\khat{X}}=\cO_{\khat{X}}(m\khat{D})$ for all positive $m$ sufficiently divisible.  In particular $\cO_X(mD)\cdot \cO_{\khat{X}}$ is reflexive for $m$ sufficiently divisible.
\end{lem}

\begin{proof} Since $\khat{D}$ is $\mu$-ample, see \cite[Lemma 6.2]{KollarMori}, $\O_{\khat{Y}}(m\khat{D})$ is $\mu$-globally generated for all positive $m$ sufficiently divisible, that is, the natural map
$$
\mu^*\mu_*\cO_{\khat{X}}(m\khat{D})\rightarrow\cO_{\khat{X}}(m\khat{D})
$$
is surjective for positive $m$ sufficiently divisible. Since $\mu$ is small, $\mu_*\cO_{\khat{X}}(m\khat{D})=\O_X(mD)$ for all integers $m$ (this is well-known, for a proof see \cite[Lemma 2.8]{ChiecchioMMPnoflips}). Thus, for all positive $m$ sufficiently divisible, we have a surjection
$$
\mu^*\cO_X(mD)\twoheadrightarrow\cO_{\khat{X}}(m\khat{D}).
$$
Notice that $\cO_{\khat{X}}\cdot\cO_X(mD)$ is isomorphic to the quotient of $\mu^*\cO_X(mD)$ by its torsion, \cite[Caution II.7.12.2]{Hartshorne}. Since $\cO_{\khat{X}}(m\khat{D})$ is torsion free, the above surjection induces a surjection
$$
\cO_{\khat{X}}\cdot\cO_X(mD)\twoheadrightarrow\cO_{\khat{X}}(m\khat{D}).
$$
On the other hand, since $\mu$ is small, for all integers $m$, $\cO_{\khat{X}}(m\khat{D})=\big(\cO_{\khat{X}}\cdot\cO_X(mD)\big){}^\vee{}^\vee$; since $\cO_{\khat{X}}\cdot\cO_X(mD)$  is torsion-free, we have a natural inclusion
$$
\cO_{\khat{X}}\cdot\cO_X(mD)\hookrightarrow\cO_{\khat{X}}(m\khat{D}).
$$
\end{proof}

\begin{lem}\label{lem.compositionofpullbacks} Let $X$ be a normal Noetherian separated scheme. Let $D$ be a Weil divisor on $X$, such that $\sR(X,-D)$ is finitely generated, and let $\sheafproj\sR(X,-D)\xrightarrow{\mu} X$. Let $f:Y\rightarrow X$ be any birational morphism factoring as $Y\xrightarrow{g}\sheafproj\sR(X,-D)\xrightarrow{\mu} X$, with $Y$ a normal Noetherian separated scheme. Then, for any positive $m$ sufficiently divisible, $f^\natural(mD)=g^\natural\mu^\natural(mD)=g^*\mu^\natural(mD)$. Therefore, $f^*(D)=g^*\mu^*(D)$.
\end{lem}

\begin{proof} This is an application of \autoref{lem.resofKX+D} which we now explain.  Consider the following chain of equalities.  The first and last equalities are by definition and the third is by \autoref{lem.resofKX+D} since $m$ is sufficiently divisible:
\[
\begin{array}{rl}
   &  \O_Y(-f^{\natural} mD)\\
 = &  \O_X(-mD) \cdot \O_Y\\
 = &  \O_X(-mD) \cdot \O_{X'} \cdot \O_Y\\
 = &  \O_{X'}(-\mu^{\natural} mD) \cdot \O_Y\\
 = &  \O_{Y}(-g^{\natural} \mu^{\natural} mD)
\end{array}
\]
This proves the first statement.  The final statement is a consequence of the fact that ${1 \over m} \mu^{\natural{mD}} = \mu^*(D)$ for $m$ sufficiently divisible by our finite generation hypothesis.
\end{proof}

\begin{lemma}
\label{lemma.FactorizingPullback}
Suppose we have a composition of birational morphisms $f' : Y' \xrightarrow{g} Y \xrightarrow{f} X$ between normal varieties and $D$ is a Weil on $X$, then $f^* D = g_* f'^* D$.
\end{lemma}
\begin{proof}
To check the identity it suffices to show $\ord_E(f^* D) = \ord_E(g_* f'^* D)$ for each prime divisor $E$ on $Y$. The generic point of each prime divisor $E$ on $Y$ gives rise to a DVR $(A_E, \fm_E)$. For any sufficiently divisible positive integer $m$, set $\nu_{E}(m D)$ the power of $\fm_E$ agreeing with $\cO_X(-mD) A_E$; so by definition $\ord_E(f^* D) = \lim \nu_{E}(m D)/m$. A similar calculation computes $f'^* D$ on $Y'$. Finally, as $g$ is birational $g_* f'^* D$ keeps the same coefficients on divisors that are not contracted by $g$, thus, $\ord_E(f^* D) = \ord_E(g_* f'^* D)$ for each prime divisor $E$ on $Y,$ as desired.
\end{proof}

We define multiplier ideals in a way which is a slight generalization of the one of \cite{DeFernexHaconSingsOnNormal}.  

\begin{definition} Let $X$ be a normal quasi-projective variety over an algebraically closed field of characteristic zero, $\Delta$ an $\RR$-Weil divisor, and $I=\prod \sJ_k^{a_k}$ a formal $\RR$-linear product of non-zero fractional ideal sheaves. The collection of the data of $X$, $\Delta$ and $I$ will be called a \emph{triple}, and it will be denoted by $(X,\Delta,I)$. We say that the triple is \emph{effective} if $\Delta\geq0$, $I=\prod \sJ_k^{a_k}$ where all the $\sJ_k$s are ideals and $a_k\geq0$ for all $k$.
\end{definition}

\begin{rmk} Notice that we do not assume that $K_X + \Delta$ is $\bR$-Cartier.
\end{rmk}

\begin{rmk} A triple is effective if and only if $(X,\Delta+I)$ is an effective pair in the sense of \cite[Definition 4.3]{DeFernexHaconSingsOnNormal}.
\end{rmk}

\begin{definition}
 \label{def.MultiplierIdealDefinition}Let $(X,\Delta,I)$ be an effective triple, and let $m$ be a positive integer. Let $f:Y\rightarrow X$ be a log resolution of the pair $(X,\O_X(-m(K_X+\Delta))+I)$, \cite[Definition 4.1 and Theorem 4.2]{DeFernexHaconSingsOnNormal}. Let $I=\prod \sJ_k^{a_k}$ be a formal product, and $\O_Y\cdot \sJ_k=\O_X(-G_k)$. We define the sheaf
$$
\mJ_m(X,\Delta,I):=f_*\O_Y(\lceil K_Y-\frac{1}{m}f^\natural(m(K_X+\Delta))-\sum a_kG_k\rceil).
$$
\end{definition}

\begin{rmk}\label{rmk:multi} The reason for this new notation is that our notation is slightly more general than the one of \cite{DeFernexHaconSingsOnNormal}. In particular, de Fernex and Hacon did not include a boundary divisor term.

This might cause some confusion since the reader might think one could absorb the divisor $\Delta$ into the ideal $I$ (indeed, what is a divisor but a formal combination of height 1 ideals).  Unfortunately, this does not yield the same object (and in particular, does not yield the usual multiplier ideal even when $K_X + \Delta$ is $\bQ$-Cartier).  The difference is that asymptotics are already built into $\bQ$-Cartier $K_X + \Delta$ whereas no asymptotics are built into $I$ in \cite{DeFernexHaconSingsOnNormal}.

In particular, let $\sI(\Delta)$ denote the formal product of ideals corresponding to $\Delta$ in the obvious way.  Then in general, $\mJ_m(X,\Delta,I)\supseteq\mJ_m(X,I \cdot \sI(\Delta))$: we have
{\setlength\arraycolsep{1pt}
\begin{eqnarray*}
\mJ_m(X,I\cdot\sI(\Delta))&=&f_*\O_Y(\lceil K_Y-\frac{1}{m}f^\natural(mK_X)-f^\natural\Delta-\sum a_kG_k\rceil)\\
&\subseteq&f_*\O_Y(\lceil K_Y-\frac{1}{m}f^\natural(mK_X)-\frac{1}{m}f^\natural(m\Delta)-\sum a_kG_k\rceil)\\
&\subseteq&f_*\O_Y(\lceil K_Y-\frac{1}{m}f^\natural(mK_X+m\Delta)-\sum a_kG_k\rceil)\\
&=&\mJ_m(X,\Delta,I).
\end{eqnarray*}}

\noindent
Here the first containment is \cite[Lemma 2.8]{DeFernexHaconSingsOnNormal} and the second is a consequence of \cite[Remark 2.11]{DeFernexHaconSingsOnNormal}.
\end{rmk}

\begin{lem} Let $(X,\Delta,I)$ be an effective triple. The sheaf $\mJ_m(X,\Delta,I)$ is a (coherent) sheaf of ideals on $X$, and its definition is independent of the choice of $f$.
\end{lem}

\begin{proof} The proof proceeds as in the proof of \cite[4.4]{DeFernexHaconSingsOnNormal}.
\end{proof}

\begin{lem} Let $(X,\Delta,I)$ be an effective triple. The set of ideal sheaves $\{\mJ_m(X,\Delta,I)\}_{m\geq1}$ has a unique maximal element.
\end{lem}

\begin{proof} For any positive integers $m,q$, $\mJ_m(X,\Delta,I)\subseteq\mJ_{mq}(X,\Delta,I)$ by \cite[Remark 3.3]{DeFernexHaconSingsOnNormal} (by the previous lemma the two ideals can be computed on a common resolution). The unique maximal ideal exists by Noetherianity.
\end{proof}

\begin{definition} Let $(X,\Delta,I)$ be an effective triple. We will call the unique maximal element of $\{\mJ_m(X,\Delta,I)\}_{m\geq1}$ the \emph{multiplier ideal} of the triple $(X,\Delta,I)$, and we will denote it by $\mJ(X,\Delta,I)$.
\end{definition}

\begin{rmk} In the case when $\Delta = 0$ we write $\mJ(X,I) = \mJ(X,0,I)$ and then our definition agrees with the one in \cite{DeFernexHaconSingsOnNormal}.
\end{rmk}

%

%

\begin{corollary}
\label{cor.MultiplierIdealDescriptionForFG}
Working in characteristic zero, suppose $\sR(-K_X - \Delta)$ is finitely generated and $X' = \sheafproj \sR(-K_X - \Delta)$ just as before.  Let $\ba$ be any ideal sheaf on $X$, then we have \mbox{$\mJ(X, \Delta, \ba^t) = \mu_* \mJ(X', \mu^* \Delta, (\ba \cdot \O_{X'})^t)$}.
\end{corollary}

\begin{proof} Let $\pi:\Proj_X\Rscr(X,D)=\Xhat\rightarrow X$ and $\Deltahat:=\pi^{-1}_*\Delta$. It is enough to show that, for every $m$ satisfying the result of \autoref{lem.resofKX+D},
$$
\pi_*\mJ((\Xhat,\Deltahat);tI\cdot\O_{\Xhat})=\mJ_m(X,\Delta,tI).
$$
Let $f:Y\rightarrow X$ be a log resolution of $(X,m\Delta+I)$ factoring through $\Xhat$. Let $I=\prod \mJ_k^{a_k}$ and let $\mJ_k\cdot\O_Y=\O(-G_k)$; since $\O_Y\cdot I=\O_Y\cdot\O_{\Xhat}\cdot I$, $f$ is a log resolution of $((\Xhat,\Deltahat);I)$. Let $g:Y\rightarrow\Xhat$. Since $(\Xhat,\Deltahat)$ is a log pair, the multiplier ideal $\mJ((\Xhat,\Deltahat);tI\cdot\O_{\Xhat})$ is
$$
\mJ((\Xhat,\Deltahat);tI\cdot\O_{\Xhat})=g_*\O_Y(\lceil K_Y-g^*(K_{\Xhat}+\Deltahat)-\sum ta_kG_k\rceil).
$$
On the other hand, for each $m\geq1$, the multiplier ideal $\mJ_m(X,\Delta;tI)$ is
$$
\mJ_m(X,\Delta,tI)=f_*\O_Y(\lceil K_Y-\frac{1}{m}f^\natural(m(K_X+\Delta))-\sum ta_kG_k\rceil).
$$
For each $m$ satisfying $\O_X(m(-K_X-\Delta))\cdot\O_{\Xhat}=\O_{\Xhat}(m(-K_{\Xhat}-\Deltahat))$, by \autoref{lem.compositionofpullbacks},
$$
f^\natural(m(K_X+\Delta))=g^*\pi^\natural(m(K_X+\Delta))=g^*(m(K_{\Xhat}+\Deltahat))=mg^*(K_{\Xhat}+\Deltahat).
$$
Therefore, for each $m$ satisfying $\O_X(m(-K_X-\Delta))\cdot\O_{\Xhat}=\O_{\Xhat}(m(-K_{\Xhat}-\Deltahat))$,
{\setlength\arraycolsep{1pt}
\begin{eqnarray*}
\mJ_m(X,\Delta,tI)&=&f_*\O_Y(\lceil K_Y-\frac{1}{m}f^\natural(m(K_X+\Delta))-\sum ta_kG_k\rceil)=\\
&=&f_*\O_Y(\lceil K_Y-g^*(K_{\Xhat}+\Deltahat)-\sum ta_kG_k\rceil)=\\
&=&\pi_*g_*\O_Y(\lceil K_Y-g^*(K_{\Xhat}+\Deltahat)-\sum ta_kG_k\rceil)=\\
&=&\pi_*\mJ((\Xhat,\Deltahat);tI\cdot\O_{\Xhat}).
\end{eqnarray*}}
\end{proof}

\begin{remark}
\label{rem.DiscAndRatAndAlterationsForMultiplierIdeals}
With the assumptions of \autoref{cor.MultiplierIdealDescriptionForFG}, it follows immediately that the jumping numbers of $\mJ(X, \Delta, \ba^t)$ are rational and without limit points.  Recall that the jumping numbers are real numbers $t_0 \geq 0$ such that $\mJ(X, \Delta, \ba^{t_0 - \varepsilon}) \neq \mJ(X, \Delta, \ba^{t_0})$ for any $\varepsilon > 0$.  It also follows that \[
\mJ(X, \Delta, \ba^t) = \bigcap_{\pi : Y \to X} \Image\big( \pi_* \O_Y(\lceil K_Y - \pi^*(K_X + \Delta) - tG \rceil) \xrightarrow{\Tr_{Y/X}} \O_X \big)
\]
where $\pi$ runs over all alterations factoring through $\mu : X' \to X$ such that $\ba \O_Y = \O_Y(-G)$ is invertible.
\end{remark}

\subsection{Finite generation of local section rings for threefolds in characteristic $p > 5$}
Of course, one might ask how often it even happens that a section ring $\sR(X, D)$ is finitely generated.  For (pseudo-)rational surface singularities of any characteristic, it is known that $D$ is always locally torsion in the divisor class group, see \cite[Theorem 17.4]{LipmanRationalSingularities}, and so obviously $\sR(D)$ is finitely generated.  However, for threefolds, rational singularities are not enough by an example of Cutkosky, \cite{CutkoskyWeilDivisorsAndSymbolic}, even if they are additionally log canonical.  Of course, in characteristic zero, the finite generation of these section rings holds for KLT $X$ of any dimension by the minimal model program \cite[Theorem 92]{KollarExercisesInBiratGeom} (and of course is closely linked with the existence of flips).

Using the recent breakthroughs on the minimal model program for threefolds in characteristic $p > 5$ \cite{BirkarExistenceOfFlipsMinimalModels,XuBasePointFreePositiveChar,CasciniTanakaXuBasePointFree,HaconXuThreeDimensionalMinimalModel,BirkarWaldronMoriFiberSpaces} one can prove finite generation of $\sR(X, D)$ in some important cases (again in dimension $3$, characteristic $p > 5$). The proof is essentially the same as it is in characteristic zero, see \cite[Exercises 108 and 109]{KollarExercisesInBiratGeom}, but we reproduce it here for the reader's convenience.

\begin{theorem}
\label{thm.FGforKLT}
Let $(X,\Delta)$ be a KLT pair of dimension $3$ with $K_X + \Delta$ $\bQ$-Cartier over an algebraically closed field $k$ of char $p > 5$. Then, for any $\Q$-divisor $D$, the algebra $\sR(X,D)$ is finitely generated.
\end{theorem}

\begin{proof}
Let $\phi : \widehat{X} \to X$ be a small $\bQ$-factorialization of $X$, which exists by \cite[Theorem 1.6]{BirkarExistenceOfFlipsMinimalModels}.  Set $\widehat{\Delta}$ and $\widehat{D}$ to be the strict transforms of $\Delta$ and $D$ on $\widehat{X}$.  Then we notice that $(\widehat{X}, \widehat{\Delta}+(1/m)\widehat{D})$ is KLT for $m \gg 0$.  By \cite[Theorem 1.3]{BirkarExistenceOfFlipsMinimalModels}, since $K_{\widehat{X}} + \widehat{\Delta}+(1/m)\widehat{D}$ is big over $X$, we see that
\[
\bigoplus_{n \geq 0} \phi_* \O_{\widehat{X}}(n (K_{\widehat{X}} + \widehat{\Delta}+(1/m)\widehat{D}) )
\]
is finitely generated.  Since $\phi$ is small, this implies that
\[
\bigoplus_{n \geq 0} \O_{X}(n (K_{X} + \Delta+(1/m)D) )
\]
is finitely generated as well (since the algebras are the same).  However, by taking a high Veronese, and recalling that $K_X + \Delta$ is $\bQ$-Cartier (and so locally, contributes nothing to finite generation) we conclude that
$
\bigoplus_{n \geq 0} \O_X(nD)
$
is finitely generated as desired.
\end{proof}



Of course this also implies that strongly $F$-regular pairs have finitely generated local section algebras since they are always KLT for an appropriate boundary by \cite{SchwedeSmithLogFanoVsGloballyFRegular}. 

%% file: ReesAlgebrasStabilization.tex
In this section we aim to prove discreteness and rationality of jumping numbers of test ideals as well as Hartshorne-Speiser-Lyubeznik-Gabber-type stabilization results under the hypothesis that the anti-canonical algebra $S = \sR(-K_R)$ is finitely generated.  We first notice that we can extend $p^{-e}$-linear maps on $R$ to $p^{-e}$-linear maps on $S$.  Note that this argument is substantially simpler than what the fourth author and K.~Tucker did to obtain similar results for finite maps in \cite{SchwedeTuckerTestIdealFiniteMaps}.

\begin{lemma}
\label{lem.InduceMapOnSymbolicAlgebra}
Suppose that $R$ is an $F$-finite normal domain, $D$ is a Weil divisor on $\Spec R$ with associated algebra $S := \sR(D)$.  Then for any $R$-linear map $\phi : F^e_* R \to R$ we have an induced $S$-linear map $\phi_S : F^e_* S \to S$ and a commutative diagram
\[
\xymatrix{
F^e_* R \ar@{^{(}->}[d] \ar[r]^-{\phi} & R\ar@{^{(}->}[d]\\
F^e_* S \ar[r]^-{\phi_S} \ar[d]_{F^e_* \rho} & S \ar[d]^{\rho} \\
F^e_* R \ar[r]_-{\phi} & R
}
\]
where $\rho$ is the projection map onto degree zero $S \to R$.
\end{lemma}
\begin{proof}
First note that we give $F^e_* S$ a $\bZ[{1 \over p^e}]$-graded structure so that our induced map $\phi_S$ will be homogeneous.
The idea is then simple, given an integer $i \geq 0$, $[F^e_* S]_i = F^e_* R(ip^e D)$.  We want to show that
\begin{equation}
\label{eq.PhiSendsDivisorToDivisor}
\phi(F^e_* R(ip^e D )) \subseteq R(iD ).
\end{equation}
But this is obvious since it holds in codimension 1 and all the sheaves are reflexive.
Finally, we simply have $\phi_S$ send $[F^e_* S]_{i / p^e}$ to zero if $i$ is not divisible by $p^e$.  This completes the proof.
\end{proof}

In fact, it is not difficult to see that every homogeneous $p^{-e}$-linear map on $S$ comes from $R$ in this way.

\begin{lemma}
Suppose $R$ is an $F$-finite normal domain, $D$ a Weil divisor on $\Spec R$, and $S = \sR(D)$, suppose we have a homogeneous map $\phi_S : F^e_* S \to S$ (again we give $F^e_* S$ the $\bZ[1/p^e]$-grading).  Then $\phi_S$ is induced from $\phi_S|_R = \phi : F^e_* R \to R$ as in \autoref{lem.InduceMapOnSymbolicAlgebra}.
\end{lemma}
\begin{proof}
Choose $F^e_* z \in [F^e_* S]_{i} = F^e_* R(p^e i D)$, invert an element $u \in S_0 = R$ to make $D$ Cartier and principal and then $z = (f/u^m) y^{p^e}$ where $y$ generates $R(iD)[u^{-1}]$ and the element $f/u^m \in R[u^{-1}] \subseteq S[u^{-1}]$ is in degree zero in $S[u^{-1}]$.  We see that $$u^{m p^e} z = f u^{m(p^e - 1)} y^{p^e} \in F^e_* R(p^e i D).$$  Hence $u^m \phi_S(F^e_* z) = \phi_S(F^e_* u^{mp^e} z) = \phi_S(F^e_* f u^{m(p^e - 1)} y^{p^e}) = y \phi_S(F^e_* f u^{m(p^e - 1)})$.  The point is that we can choose the same $y$ regardless of the choice of $z$.  Hence $\phi_S$ is completely determined by $\phi_S|_{R}$.
\end{proof}

\autoref{lem.InduceMapOnSymbolicAlgebra} is key in the following proposition which lets us relate $p^{-e}$-linear maps in general on $R$ and $S$.
\begin{proposition}
\label{prop.MainCommutativeDiagramOfSymbolicRees}
Suppose that $R$ is an $F$-finite normal domain, $D$ is a Weil divisor, and the algebra $S = \sR(D)$ is finitely generated (and in particular an $F$-finite Noetherian ring).  Further suppose that $G$ is an effective Weil divisor on $\Spec R$ with pullback $h^*G = G_S$ on $\Spec S$.
Then we have a commutative diagram
\[
\xymatrix@C=50pt{
\Hom_S(F^e_* S(G_S), S) \ar@{^{(}->}[r] \ar[d]_{\nu} & \Hom_S(F^e_* S, S) \ar[r]^-{E_S = \textnormal{eval@1}} \ar[d]_{\gamma} & S \ar[d]^{\rho} \\
\Hom_R(F^e_* R(G), R) \ar@{^{(}->}[r] & \Hom_R(F^e_* R, R) \ar[r]_-{E_R = \textnormal{eval@1}} & R.
}
\]
Here the map $\rho$ is projection onto the $0$th coordinate $[S]_0 = R$ and $\gamma$ is the map which restricts $\phi \in \Hom_S(F^e_* S, S)$ to $F^e_* R = [F^e_* S]_0$ and then projects onto $[S]_0 = R$.  Furthermore the maps $\gamma$ and $\nu$ are surjective.
\end{proposition}
\begin{proof}
We first handle the commutativity.  Given $\psi \in \Hom_S(F^e_* S, S)$, we see that $\rho(E_S(\psi)) = \rho(\psi(1))$.  On the other hand $E_R(\gamma(\psi)) = \rho(\psi(1))$ as well.  Hence we have commutativity of the right square.  The commutativity of the left square is obvious since $G_S$ is pulled back from $\Spec R$

To see that $\gamma$ is surjective, for any $\phi \in \Hom_R(F^e_* R, R)$ construct $\phi_S$ as in \autoref{lem.InduceMapOnSymbolicAlgebra}.  Obviously $\gamma(\phi_S) = \phi$.  Similarly, \autoref{lem.InduceMapOnSymbolicAlgebra} implies the surjectivity of the map $\nu$.
\end{proof}

As an immediate corollary we obtain a stabilization result similar to Hartshorne-Speiser-Lyubeznik-Gabber.

\begin{corollary}
\label{cor.HartshorneSpeiserStabilization}
Suppose that $R$ is an $F$-finite normal domain and that $B \geq 0$ is a Weil divisor.  If the anti-log-canonical algebra $\sR(-K_R-B)$ is finitely generated, then the image of the evaluation-at-1 map $\Hom_R(F^e_* R( (p^e - 1) B), R) \to R$ stabilizes for $e \gg 0$.
\end{corollary}
\begin{proof}
Set $S = \sR(-K_R-B)$ and consider the diagram of \autoref{prop.MainCommutativeDiagramOfSymbolicRees}.  Since $K_S + h^* B$ is Cartier, we see that the images of $$E_S^e : \Hom_S(F^e_* S( (p^e - 1)h^*B), S) \xrightarrow{\text{eval@1}} S$$ stabilize, see for instance \cite{FujinoSchwedeTakagiSupplements}.  But then since $\nu$, as in \autoref{prop.MainCommutativeDiagramOfSymbolicRees}, surjects, we see that the image of
\begin{equation}\label{eq:image}
\Hom_S(F^e_* S( (p^e - 1)B), R) \xrightarrow{\nu} F^e_* \Hom_R(F^e_* R( (p^e - 1) B), R) \xrightarrow{\text{eval@1}} R
\end{equation}
coincides with that of $\Hom_R(F^e_* R( (p^e - 1) B), R)  \xrightarrow{E_R^e} R$.  But the image of \eqref{eq:image} is the same as the image of
\[
\Hom_S(F^e_* S( (p^e - 1)h^*B), S) \xrightarrow{E_S^e} S \xrightarrow{\rho} R.
\]
However the $E_S^e$ have stable image as we have already observed and the result follows.
\end{proof}

Later in \autoref{thm.HartshorneSpeiserStabilization}, we will obtain the same result for $\bQ$-divisors whose Weil-index is not divisible by $p$.  For now though, we move on to discreteness and rationality of $F$-jumping numbers, generalizing \cite[Theorem 6.4]{KatzmanSchwedeSinghZhang} from the case of a graded ring $R$.

\begin{theorem}
\label{thm.DiscAndRatOfFjumping}
Suppose that $R$ is a normal domain and $\Delta \geq 0$ is an effective $\bQ$-divisor such that $\sR(-K_R - \Delta)$ is finitely generated.  Then for any ideal $\ba \subseteq R$ the $F$-jumping numbers of $\tau(R, \Delta, \ba^t)$ are rational and without limit points.
\end{theorem}
\begin{proof}
First let $R \subseteq R'$ be a separable extension of normal $F$-finite domains corresponding to a map of schemes $\Spec R' = X' \xrightarrow{\nu} X = \Spec R$ such that $\nu^* \Delta$ is an integral divisor (this is easy, the idea is to simply take roots of generators of DVRs, if one has to take a $p$th root, use Artin-Schreyer type equations see \cite[Lemma 4.5]{BlickleSchwedeTuckerTestAlterations}).  Let $\Tr : K(X') \to K(X)$ be the trace map and then recall that $\Tr\big( \nu_* \tau(X', \nu^* \Delta - \Ram_{\nu}, (\ba R')^t) \big) = \tau(X, \Delta, \ba^t)$ by the main result of \cite{SchwedeTuckerTestIdealFiniteMaps}.  It immediately follows that if the $F$-jumping numbers of test ideal $\tau(X', \nu^* \Delta-\Ram_{\nu}, \ba^t)$ are discrete and rational, so are the $F$-jumping numbers of $\tau(X, \Delta, \ba^t)$.  Additionally, by adding a Cartier divisor $H$ to $\Delta$, we can assume that $\nu^* \Delta - \Ram_{\nu}$ is effective since $\tau(X, \Delta+H, \ba^t) = \tau(X, \Delta, \ba^t) \otimes \O_X(-H)$ by \autoref{lem.BasicTestIdealProperties}(d).  Finally, note that $$-K_{R'} - \nu^* \Delta + \Ram_{\nu} = \nu^*(-K_R - \Delta)$$ and so $\sR(-K_{R'} - \nu^* \Delta + \Ram_{\nu})$ is finitely generated by \autoref{lem.BasicPropertiesOfLocalSectionAlgebras}.  The upshot of this entire paragraph is of course that we may now without loss of generality assume that $\Delta$ is an integral effective divisor.

Next choose $c \in R$ that is a test element for both $((R, \Delta)$ and $S := \sR(-K_R - \Delta)$.  The choice of such a $c$ is easy, simply choose a test element so that additionally $-K_R - \Delta$ is Cartier on $X \setminus V(c)$.  Away from $V(c)$, $S$ looks locally like $R[t]$ which will certainly be strongly $F$-regular over wherever $R$ is strongly $F$-regular \cite{HochsterHunekeFRegularityTestElementsBaseChange}.  Let $H$ be the Cartier divisor corresponding to $c$ and consider the commutative diagram.
\[
{
\scriptsize
\xymatrix@C=30pt{
(F^e_* (\ba S)^{\lceil t(p^e - 1) \rceil}) \cdot \Hom_S(F^e_* S(h^* (p^e - 1) \Delta + h^* H), S) \ar@{^{(}->}[r] \ar[d]_{\nu} & \Hom_S(F^e_* S, S) \ar[r]^-{ \textnormal{eval@1}} \ar[d]_{\gamma} & S \ar[d]^{\rho} \\
(F^e_* \ba^{\lceil t(p^e - 1) \rceil}) \cdot \Hom_R(F^e_* R((p^e - 1)\Delta + H), R) \ar@{^{(}->}[r] & \Hom_R(F^e_* R, R) \ar[r]_-{\textnormal{eval@1}} & R.
}
}
\]
The sum over $e > 0$ of the images of the bottom rows is equal to $\tau(R, \Delta, \ba^t)$ and the sum over $e > 0$ of the images of the top row is equal to $\tau(S, h^* \Delta, (\ba S)^t)$.  Since $\nu$ surjects by \autoref{prop.MainCommutativeDiagramOfSymbolicRees}, we immediately see that $\rho(\tau(S, h^* \Delta, (\ba S)^t)) = \tau(R, \Delta, \ba^t)$.  But now observe that $K_S + h^* \Delta \sim 0$ by \autoref{lem.PullBackDivisorsOnSymbolicRees}.  But then the $F$-jumping numbers of $\tau(S, h^* \Delta, (\ba S)^t)$ are discrete and rational by \cite{SchwedeTuckerNonPrincipalIdeals}.  The result follows.
\end{proof}

We immediately obtain the following using the aforementioned breakthroughs in the MMP.

\begin{corollary}
\label{cor.DiscretenessForKLT}
Suppose that $R$ is strongly $F$-regular, of dimension $3$, and of finite type over an algebraically closed field of characteristic $p > 5$.  Then the $F$-jumping numbers $\tau(R, \Delta, \ba^t)$ are rational and without limit points for any choice of $\bQ$-divisor $\Delta$ and ideal $\ba$.
\end{corollary}
\begin{proof}
Since $R$ is strongly $F$-regular, there exists a divisor $\Gamma \geq 0$ so that $K_R + \Gamma$ is $\bQ$-Cartier and so that $(R, \Gamma)$ is KLT by \cite{SchwedeSmithLogFanoVsGloballyFRegular}.  The result then follows from \autoref{thm.FGforKLT} and \autoref{thm.DiscAndRatOfFjumping}.
\end{proof}

Of course, we also obtain discreteness and rationality of $F$-jumping numbers $\tau(R, \Delta, \ba^t)$ for any $R$, a 3-dimensional ring of finite type over an algebraically closed field $k$ of characteristic $p > 5$ such that there exists a $\Gamma \geq 0$ so that $(R, \Gamma)$ is KLT.

\subsection{A more general Hartshorne-Speiser type result}

In \autoref{cor.HartshorneSpeiserStabilization}, we used a compatibility of the formation of Rees algebras to prove that the images of $\Hom_R(F^e_* R, R) \to R$ stabilize for large $e$ if $S := \sR(-K_R)$ is finitely generated.  In this short section, we generalize this result to the case of $\bQ$-divisors, at least whose Weil-index is not divisible by $p$.  As an alternate strategy, one could try to prove compatibilities analogous to \autoref{prop.MainCommutativeDiagramOfSymbolicRees} for Rees algebras of $\bQ$-divisors.  Unfortunately this gets quite messy.  Instead we take a different approach utilizing ${\sheafproj} S$.  We first prove the result for $\bQ$-Gorenstein varieties and then we handle the finitely generated case via the small map $\mu : X' \to X$.

We do restrict ourselves to the case where the Weil-index of $K_X + \Delta$ is not divisible by $p$.  We realize that the methods we discuss here can apply to more general situations but there are then several potential competing definitions for what the stable image should be.

\begin{proposition}
\label{prop.HartshorneSpeiserStabilizeQGor}
Suppose that $(R, \Delta \geq 0)$ is a pair such that $K_R + \Delta$ is $\bQ$-Cartier.  Suppose that the Weil index of $K_R + \Delta$ is not divisible by $p$ and that $(p^e - 1)(K_R + \Delta)$ is an integral Weil divisor.  Then
\[
\sigma_{ne}(R, \Delta) := \Image\Big( \Hom_R(F^{ne}_* R( (p^{ne} - 1)\Delta), R) \xrightarrow{E^{ne} = \textnormal{eval@1}} R\Big)
\]
stabilizes for large $n$.
\end{proposition}
\begin{proof}
Fix $m > 0$ so that $m(K_R + \Delta)$ is a Cartier divisor.
The main idea is that module $\Hom_R(F^{ne}_* R( (p^{ne} - 1)\Delta), R)$ only takes on the values of finitely many sheaves, at least up to twisting by line bundles (in particular, multiples of $R(m(K_R + \Delta))$.  We also take advantage of the fact that it is sufficient to show that the images stabilize partially up the chain.

\begin{claim}
Fix $n_0 > 0$ and consider $n \geq n_0$.  Then $$\Hom_R(F^{ne}_* R( (p^{ne} - 1)\Delta), R) \xrightarrow{E^{ne} = \textnormal{eval@1}} R$$ factors through $$\Hom_R(F^{n_0e}_* R( (p^{n_0 e} - 1)\Delta), R) \xrightarrow{E^{n_0e} = \textnormal{eval@1}} R.$$  Hence it is sufficient to show that the images of $\Hom_R(F^{ne}_* R( (p^{ne} - 1)\Delta), R) $ in $$\Hom_R(F^{n_0e}_* R( (p^{n_0e} - 1)\Delta), R)$$ stabilize.
\end{claim}
\begin{proof}[Proof of claim]
One simply notices that
\[
p^{(n-n_0)e} (p^{n_0e} - 1) \Delta \leq (p^{ne} - 1)\Delta
\]
and hence $(R( (p^{ne} - 1) \Delta))^{1/p^{ne}}$ contains $(R( (p^{n_0e} - 1) \Delta))^{1/p^{n_0e}}$.  Thus the claimed factorization occurs simply by restriction of scalars.
\end{proof}
We continue on with the main proof.  Note that $p^t \mod m$ is eventually periodic.  Then choose a linear function $\theta(a) = ca + r$, for $c > 0, r \geq 0$, $c, r\in \bZ$, such that $(p^{\theta(a)e} - 1) \mod m$ is constant.  Set
\[
M = R( ((1-p^{re})\mod{m})(K_R + \Delta)) = R( ((1-p^{\theta(0)e})\mod{m})(K_R + \Delta))
\]
and note that for any $a \geq 0$,
\[
\begin{array}{rl}
& \Hom_R( F^{\theta(a)e}_* R((p^{\theta(a)e} - 1)\Delta), R) \\
\cong & F^{\theta(a)e}_* R( (1-p^{\theta(a)e})(K_R + \Delta))\\
\cong & F^{\theta(a)e}_* \big(R( ((1 - p^{\theta(a)e})\mod{m})(K_R + \Delta)) \otimes R(\lfloor {1 - p^{\theta(a)e}\over m}\rfloor m(K_R + \Delta))\big)\\
\cong & F^{\theta(a)e}_* \big(M \otimes R(\lfloor {1 - p^{\theta(a)e}\over m}\rfloor m(K_R + \Delta))\big)
\end{array}
\]

By inverting an element of $R$ if necessary, we may assume that $m(K_R + \Delta) \sim 0$.  Thus by utilizing this, we have maps
\[
\ldots \xrightarrow{t_a} F^{\theta(a) e}_* M \xrightarrow{t_{a-1}}  F^{\theta(a-1) e}_* M \to \ldots \xrightarrow{t_1} F^{\theta(1) e}_* M \xrightarrow{t_0} F^{\theta(0) e}_* M.
\]
If these maps are Frobenius pushforwards of each other, \ie $F^{ce}_* t_{a-1} = t_a$ (or at least up to a unit), then we can apply the standard Hartshorne-Speiser-Lyubeznik-Gabber theorem \cite{Gabber.tStruc} to conclude that the images stabilize in $F^{\theta(0) e}_* M$.  But this may be checked in codimension 1 (since all sheaves are reflexive and so maps between them are determined in codimension 1).  However, after localizing to reduce to codimension 1, all the complicated twisting we have done is irrelevant.   Furthermore, in codimension $1$, $R$ is Gorenstein with $K_R \sim 0$ and $\Delta$ is $\bQ$-Cartier with index not divisible by $p > 0$ (since its Weil index was not divisible by $p > 0$).  Our chain of maps then just turns into
\[
\xymatrix{
\Hom_R( F^{\theta(a)e}_* R((p^{\theta(a)e} - 1)\Delta), R) \ar[r] & \Hom_R( F^{\theta(a-1)e}_* R((p^{\theta(a-1)e} - 1)\Delta), R) \to \ldots \\
F^{\theta(a)e}_* \omega_R( -p^{\theta(a)e}(K_R) - (p^{\theta(a)e} -1) \Delta) \ar@{<->}[u]_{\sim} \ar[r]_-{\Tr} & F^{\theta(a-1)e}_* \omega_R( -p^{\theta(a-1)e}(K_R) - (p^{\theta(a-1)e} -1) \Delta) \ar@{<->}[u]_{\sim}\\
F^{(ca+r)e}_* \omega_R( - (p^{(ca+r)e} -1) \Delta) \ar@{<->}[u]_{\sim} \ar[r]_-{\Tr} & F^{(c(a-1)+r)e}_* \omega_R( - (p^{(c(a-1)+r)e} -1) \Delta) \ar@{<->}[u]_{\sim}.
}
\]
The bottom horizontal map is then obtained via
\[
\begin{array}{rl}
& F^{(ca+r)e}_* \omega_R( - (p^{(ca+r)e} -1) \Delta) \\
\hookrightarrow & F^{(ca+r)e}_* \omega_R( - p^{ce}(p^{(c(a-1)+r)e} -1) \Delta) \\
\xrightarrow{\Tr} & F^{(c(a-1)+r)e}_* \omega_R( -(p^{(c(a-1)+r)e} -1) \Delta).
\end{array}
\]
Note the inclusion $\hookrightarrow$ can be identified with multiplication by a defining equation for
\[
- p^{ce}(p^{(c(a-1)+r)e} -1) \Delta + (p^{(ca+r)e} -1) \Delta = (p^{ce}-1)\Delta.
\]
This is independent of $a$ and so the maps in our chain are really the same, up to pushforward, as claimed.

Note that this completes the proof.  Even though we only proved stabilization of images for a subset of $ne > 0$, these images are descending and our subset is infinite.
\end{proof}

Now we are in a position to prove \autoref{cor.HartshorneSpeiserStabilization} in the more general situation.

\begin{theorem}
\label{thm.HartshorneSpeiserStabilization}
Suppose that $R$ is an $F$-finite normal domain and that $B \geq 0$ is a $\bQ$-divisor with Weil-index not divisible by $p$.  If the anti-log-canonical algebra $\sR(-K_R-B)$ is finitely generated, then the image of the evaluation at 1 map $\Hom_R(F^e_* R( (p^e - 1) B), R) \to R$ stabilizes for $e$ sufficiently divisible.
\end{theorem}
\begin{proof}
For this proof, we will phrase our evaluation-at-1 maps in terms of the trace
$$F^e_* R( (1-p^e)(K_R + B)) \to R.$$  We thus fix an $e > 0$ so that $(p^e - 1)(K_R + B)$ is an integral Weil divisor.

Let $\mu : X' =\sheafproj \sR(-K_R-B) \to X = \Spec R$ be as before.  We observe that $\mu^*(-K_R-B)$ is $\bQ$-Cartier and also still has Weil-index not divisible by $p$ by \autoref{lem.BasicPropertiesOfLocalSectionAlgebras}.  Hence the images
{
\begin{equation}
\label{eq.ChainOfImagesOnY}
\begin{array}{rl}
\ldots \to & F^{(n+1)e}_* \O_{X'}( (1-p^{(n+1)e})(K_{X'} + \mu^* B)) \\
\to & F^{ne}_* \O_{X'}( (1-p^{ne})(K_{X'} + \mu^* B)) \\
\to & \ldots \\
\to & F^e_* \O_{X'}( (1-p^e)(K_{X'} + \mu^*B)) \\
\to & \O_{X'}
\end{array}
\end{equation}
}
stabilize by \autoref{prop.HartshorneSpeiserStabilizeQGor}.  In fact, the same argument even shows that the images even stabilize in any finite stage, such as in $F^{ne}_* \O_{X'}( (1-p^{ne})(K_{X'} + \mu^* B))$.  However, the terms and maps in this chain take on finitely many values up to twisting by (large) Cartier multiples of $-(K_{X'} + \mu^* B)$ (as argued in \autoref{prop.HartshorneSpeiserStabilizeQGor}).  Our goal is to thus show that these images stabilize after pushing forward by $\mu$.
\begin{claim}
If one applies $\mu_*$ to \autoref{eq.ChainOfImagesOnY}, obtaining
\[
\ldots \to \mu_* F^{ne}_* \O_{X'}( (1-p^{ne})(K_{X'} + \mu^* B))
\to \ldots
\to \mu_* F^e_* \O_{X'}( (1-p^e)(K_{X'} + \mu^*B))
\to \mu_* \O_{X'}
\]
then the chain of images in $\O_X = \mu_* \O_{X'}$ still stabilizes.
\end{claim}
\begin{proof}[Proof of claim]
Choose $d > 0$ so that
\begin{equation}
\label{eq.StabledImage}
\Image\Big(F^{(n+d)e}_* \O_{X'}( (1-p^{(n+d)e})(K_{X'} + \mu^* B)) \to F^{ne}_* \O_{X'}( (1-p^{ne})(K_{X'} + \mu^* B))\Big)
\end{equation}
is equal to the stable image, which we denote by $F^{ne}_* \sigma_{ne}$, for all $n \geq 0$ (note by \autoref{prop.HartshorneSpeiserStabilizeQGor} there are finitely many conditions).  Observe that there are only finitely many $\sigma_{ne}$ up to twisting by (large) multiples of $-(K_{X'}+\mu^*B)$.  The fact that $-(K_{X'} + \mu^* B)$ is ample implies that there exists an $n_0 \geq 0$ so that for any $n \geq n_0$,
\[
\begin{array}{rlcl}
& F^{ne}_* \mu_* \sigma_{ne}\\
= & \Image\Big(F^{(n+d)e}_* \mu_* \O_{X'}( (1-p^{(n+d)e})(K_{X'} + \mu^* B)) & \to & F^{ne}_* \mu_* \O_{X'}( (1-p^{ne})(K_{X'} + \mu^* B))\Big)\\
 = & \Image\Big(F^{(n+d)e}_* \mu_* \O_{X'}( (1-p^{(n+d)e})(K_{X'} + \mu^* B)) & \to & F^{ne}_* \mu_* \sigma_{ne} \Big)\\
= & \Image\Big(F^{(n+2d)e}_* \mu_* \O_{X'}( (1-p^{(n+2d)e})(K_{X'} + \mu^* B)) & \to & F^{ne}_* \mu_* \O_{X'}( (1-p^{ne})(K_{X'} + \mu^* B))\Big).\\
\end{array}
\]
But the map $F^{(n+2d)e}_* \mu_* \O_{X'}( (1-p^{(n+2d)e})(K_{X'} + \mu^* B)) \to F^{ne}_* \mu_* \O_{X'}( (1-p^{ne})(K_{X'} + \mu^* B))$ factors through
\[
F^{(n+2d)e}_* \mu_* \O_{X'}( (1-p^{(n+2d)e})(K_{X'} + \mu^* B)) \to F^{(n+d)e}_* \mu_* \O_{X'}( (1-p^{(n+d)e})(K_{X'} + \mu^* B))
\]
which has image $\mu_* \sigma_{(n+d)e}$ by our assumption that \autoref{eq.StabledImage} is the stable image, applied to the choice of $n = n+d$.  It follows that $F^{(n+d)e}_* \mu_* \sigma_{(n+d)e} \to F^{ne}_* \mu_* \sigma_{ne}$ surjects for all $n$ and so by composition, $F^{(n+c)e}_* \mu_* \sigma_{(n+c)e} \to F^{ne}_* \mu_* \sigma_{ne}$ surjects for every $n \geq n_0$ and every $c \geq d$ (note $n$ does not depend on $c$).  Thus since $F^{(n+c)e}_* \mu_* \sigma_{(n+c)e}$ is the image of $\mu_* F^{(n+d+c)e}_*\O_{X'}( (1-p^{(n+d+c)e})(K_{X'} + \mu^* B))$, we see that
\[
\begin{array}{rlcl}
& \Image\Big(F^{(n+c)e}_* \mu_* \O_{X'}( (1-p^{(n+c)e})(K_{X'} + \mu^* B)) & \to & F^{ne}_* \mu_* \O_{X'}( (1-p^{ne})(K_{X'} + \mu^* B))\Big) \\
= & \mu_* \Image\Big(F^{(n+c)e}_* \O_{X'}( (1-p^{(n+c)e})(K_{X'} + \mu^* B)) & \to & F^{ne}_* \O_{X'}( (1-p^{ne})(K_{X'} + \mu^* B))\Big)\\
= & \mu_* \sigma_{ne}
\end{array}
\]
for all $n \geq n_0$ and all $c \geq 2d$.  This clearly proves the desired stabilization.
\end{proof}
We return to the proof of theorem.  But this is trivial once we observe that
\[
F^{ne}_* \mu_* \O_{X'}( (1-p^{ne})(K_{X'} + \mu^* B)) = F^{ne}_* \O_X( (1-p^{ne})(K_X + B))
\]
since $\mu$ is small.
Hence the proof is complete.
\end{proof} 

%% file: PositivityStabilization.tex
In the previous section we showed the discreteness and rationality of $F$-jumping numbers via passing to the local section algebra (\ie a symbolic Rees algebra) where we already knew discreteness and rationality.  In this section, we recover the same discreteness result in the projective setting by using the methods of \cite{ChiecchioUrbinatiAmpleWeilDivisors} which allow us to apply asymptotic vanishing theorems to Weil divisors.  Indeed we first prove global generation results for test ideals by employing similar methods to \cite{MustataNonNefLocusPositiveChar}.

\begin{setting}
Let $X$ be a normal projective variety of characteristic $p > 0$, $\Delta \geq 0$ an effective Weil $\QQ$-divisor, $\fra$ an ideal sheaf on $X$ and $t \in \QQ$. We make no assumptions about $K_X + \Delta$ being $\QQ$-Cartier.
\end{setting}

Assume $\sG$ is a line bundle such that there are global sections $x_1, \ldots, x_m \in H^0(X, \fra \otimes \cG)$ which globally generate $\fra \otimes \cG$ and then let $\Sym^c(x_1, \ldots, x_m)$ denote the $c$th symmetric power of the vector space $\langle x_1, \ldots, x_m \rangle$.  Observe that $\Sym^c(x_1, \ldots, x_l) \subseteq H^0(X, \fra^c \otimes \cG^c)$ globally generates $\fra^c \otimes \cG^c$.  Thus we have a surjection of sheaves
\[
\Sym^c(x_1, \ldots, x_t) \otimes_k \cG^{-c} \to \fra^c.
\]

\begin{lem}\label{lem:FiniteSumNonPrinc}
If the Weil-index of $\Delta$ is not divisible by $p$, and $t=a/b$, with $p$ not dividing $b$, then there is a Cartier divisor $H$ and a finite set of integers $e_1, \ldots, e_s \gg 0$ such that
$(p^{e_i} - 1) \Delta$ is integral, $(p^{e_i} - 1)t \in \ZZ$ and  $\tau(X, \Delta, \fra^t)$ equals
\[
\sum_{i=1}^s  \Image\left({\scriptsize \xymatrix{F_*^{e_i} \Sym^{t(p^{e_i}-1)}(x_1, \ldots, x_m) \otimes_k \cG^{-t(p^{e_i}-1)} \otimes  \cO_X( (1-p^{e_i}) (K_X+\Delta) - H) \ar@{->>}[d] \ar[r] &  \cO_X \\
 F_*^{e_i} \fra^{t(p^{e_i} - 1)} \cdot \cO_X( (1-p^{e_i}) (K_X+\Delta) - H) \ar[ur]_-{\Tr^{e_i}} & }}\right).
\]
\end{lem}
At some level this result is obvious.  The only technicalities involve showing that the various rounding choices we make all give the same result in the end (since we can absorb any differences into the test element -- a local generator of $H$).  We include a complete proof but we invite the reader to skip over it if they are already familiar with this type of argument.
\begin{proof}
The statement in the end is local and so trivializing $\cG$ it suffices to show that
\[
\tau(X, \Delta, \fra^t) = \sum_{i = 1}^s \Tr^{e_i} F_*^{e_i} \big(\fra^{t(p^{e_i} - 1)} \cdot \cO_X( (1-p^{e_i}) (K_X+\Delta) - H)\big).
\]
Pick an effective Cartier divisor $H_0$ corresponding to the vanishing locus of a test element so that for any integer $e_0 > 0$
\[
\tau(X, \Delta, \fra^t) = \sum_{e \geq e_0} \Tr^e F_*^e \fra^{\lceil t p^e  \rceil} \cO_X ( \lceil K_X - p^e(K_X + \Delta) - H \rceil )
\]
for all Cartier $H \geq H_0$, c.f., \cite[Prop. 3.6]{SchwedeTuckerNonPrincipalIdeals}, \cite[Def.-Prop. 3.3]{BlickleSchwedeTakagiZhang}. This equality also holds for any $\QQ$-divisor $H \geq H_0$ as one can always pick a Cartier $H'$ so that $H' \geq H \geq H_0$ and one obtains inclusions
\begin{eqnarray*}
\tau(X,\Delta, \fra^t) &=& \sum_{e \geq e_0} \Tr^e F_*^e \fra^{\lceil t p^e  \rceil}\cO_X ( \lceil K_X - p^e(K_X + \Delta) - H_0 \rceil ) \\
&\supset& \sum_{e \geq e_0} \Tr^e F_*^e \fra^{\lceil t p^e  \rceil}\cO_X ( \lceil K_X - p^e(K_X + \Delta) - H \rceil ) \\
 &\supset& \sum_{e \geq e_0} \Tr^e F_*^e \fra^{\lceil t p^e  \rceil}\cO_X ( \lceil K_X - p^e(K_X + \Delta) - H ' \rceil ) = \tau(X,\Delta, \fra^t).\end{eqnarray*}
 Next consider the claim which will allow us restrict to those $e$ which are multiples of $e_0$.
 \begin{claim}
For any Weil divisor $H \geq H_0$, there exists a Cartier divisor $G$ such that for any integer $0 \leq b \leq e_0-1$ and for any integer $m > 0$ we have
\[
\begin{array}{rl}
& \Tr^{{e_0 m + b}} F_*^{e_0 m + b} \fra^{\lceil t p^{e_0 m + b} \rceil} \cO_X ( \lceil K_X - p^{e_0 m + b}(K_X + \Delta) - H - G \rceil ) \\
\subseteq & \Tr^{{e_0 m}} F_*^{e_0 m} \fra^{\lceil t p^{e_0 m } \rceil} \cO_X (  \lceil K_X - p^{e_0 m} (K_X + \Delta) - H \rceil\rceil ).
\end{array}
\]
\end{claim}
\begin{proof}
To prove the claim, first note that by \cite[Lemma 4.6]{SchwedeFAdjunction} (among many other places), if $d \in \fra^{ l p^{e_0}} \subseteq \fra^{l (p^{b}-1)}$ then
\[
d \fra^{ \lceil t p^{e_0 m + b} \rceil } \subseteq \fra^{\lceil t p^{e_0 m} \rceil + l (p^b-1) } \subseteq (\fra^{ \lceil t p^{e_0 m} \rceil })^{[p^b]}.
\]
where again the $l$ is an upper bound for the number of generators of $\fra$ (note that $d$ works for any $b \leq e_0 - 1$).  Set then $G = \textnormal{div}(d) + p^{e_0} H$ and notice that
\[
\begin{array}{rl}
& \Tr^b F_*^b \fra^{\lceil t p^{e_0 m + b} \rceil} \cO_X ( \lceil K_X - p^{e_0 m + b}(K_X + \Delta) - H - G \rceil ) \\
\subseteq & \Tr^b F_*^b  \big( d \fra^{\lceil t p^{e_0 m + b} \rceil} \cO_X ( \lceil K_X - p^{e_0 m + b}(K_X + \Delta) \rceil ) \otimes \cO_X(-p^b H) \big) \\
\subseteq & \Tr^b F_*^b  \big( (\fra^{ \lceil t p^{e_0 m} \rceil })^{[p^b]} \cO_X ( \lceil K_X - p^{e_0 m + b}(K_X + \Delta ) \rceil ) \otimes \cO_X(-p^b H) \big) \\
\subseteq &  \fra^{ \lceil t p^{e_0 m} \rceil } \cO_X ( \lceil K_X - p^{e_0 m}(K_X + \Delta ) \rceil ) \otimes \cO_X(-H) \\
\subseteq & \fra^{ \lceil t p^{e_0 m} \rceil } \cO_X ( \lceil K_X - p^{e_0 m}(K_X + \Delta ) - H\rceil )
\end{array}
\]
Now applying $\Tr^{e_0 m} F_*^{e_0 m}$ proves the claim.
\end{proof}
Now we return to the proof of the lemma.
The claim and our previous work implies that for a sufficiently large Cartier divisor $H \geq H_0$ and $G$ depending on $H$, we have that
\[
\begin{array}{rl}
\tau(X,\Delta, \ba^t) = & \sum_{e \geq e_0} \Tr^e F_*^e \fra^{\lceil t p^e \rceil} \cO_X ( \lceil K_X - p^e(K_X + \Delta) - H - G \rceil ) \\
\subseteq & \sum_{e_0 | e} \Tr^e F_*^e  \fra^{\lceil t p^e \rceil} \cO_X ( \lceil K_X - p^e(K_X + \Delta ) - H \rceil )\\
\subseteq & \sum_{e \geq e_0} \Tr^e F_*^e  \fra^{\lceil t p^e \rceil} \cO_X ( \lceil K_X - p^e(K_X + \Delta ) - H \rceil ) \\
= & \tau(X, \Delta, \fra^t)
\end{array}
\]
and therefore that $\tau(X, \Delta, \fra^t) = \sum_{e_0 | e} \Tr^e F_*^e  \fra^{\lceil t p^e \rceil} \cO_X ( \lceil K_X - p^e(K_X + \Delta) - H \rceil )$.

Pick a Cartier divisor $H'$ so that $H' \geq H_0 + \Delta + \lceil t \rceil \textnormal{div}(d) \geq H_0$ where $d \in \fra$.  We have that
\begin{eqnarray*}
\tau(X, \Delta, \fra^t) &=& \sum_{e_0|e} \Tr^e F_*^e \fra^{\lceil tp^e \rceil} \cO_X ( \lceil K_X - p^e(K_X + \Delta) - H' \rceil ) \\
&\subseteq & \sum_{e_0|e} \Tr^e F_*^e \fra^{\lceil t(p^e - 1) \rceil}\cO_X ( \lceil K_X - p^e(K_X + \Delta ) + \Delta - H' \rceil ) \\
& = &\sum_{e_0|e} \Tr^e F_*^e\fra^{\lceil t(p^e - 1) \rceil} \cO_X ( \lceil (1-p^e) (K_X + \Delta ) - H' \rceil )\\
& \subseteq & \sum_{e_0|e} \Tr^e F_*^e\fra^{\lceil t(p^e - 1) \rceil} \cO_X ( \lceil (1-p^e) (K_X + \Delta ) - H_0 - \Delta - \lceil t \rceil \textnormal{div}(d) \rceil ) \\
& \subseteq & \sum_{e_0|e} \Tr^e F_*^e\fra^{\lceil t(p^e - 1) \rceil + \lceil t \rceil} \cO_X ( \lceil K_X - p^e(K_X + \Delta) - H_0 \rceil ) \\
& \subseteq & \sum_{e_0|e} \Tr^e F_*^e\fra^{\lceil tp^e \rceil} \cO_X ( \lceil K_X - p^e(K_X + \Delta) - H_0 \rceil ) \\
& \subseteq & \tau(X, \Delta, \fra^t).
\end{eqnarray*}
In particular
\[
\tau(X, \Delta, \fra^t) = \sum_{e_0|e} \Tr^e F_*^e\fra^{\lceil t(p^e - 1) \rceil} \cO_X ( \lceil (1-p^e) (K_X + \Delta ) - H' \rceil ).
\]
Since the Weil-index of $\Delta$ is not divisible by $p$, $(1-p^e) (K_x + \Delta + t D)$ is an integral divisor for $e$ sufficiently divisible.  Hence by choosing our $e_0$ sufficiently divisible and noting that our scheme is Noetherian and so the above sum is finite, we obtain our desired result.
\end{proof}

\begin{rmk}
While it is certainly possible to generalize this to handle $t \in \RR$ or to handle $\Delta$ such that $(p^e - 1)(K_X + \Delta)$ is not integral, those generalizations are not the ones we need.  In particular we will need a power of $K_X + \Delta$ times a locally free sheaf.
\end{rmk}

\begin{thm}
\label{thm.GGofNonPrincipalMult}
Suppose $X$ is normal and projective, $\Rscr(X,-K_X-\Delta)$ is finitely generated and $\Delta$ has Weil-index not divisible by $p$ and fix $t > 0$.
There exists a Cartier divisor\footnote{See \autoref{rem.EffectiveDivisorForVanishing} for a discussion of how to choose $L$ effectively.} $L$
such that
$$
\tau(X,\Delta, \fra^w)\otimes\cO_X(L)
$$
is globally generated when $0\leq w=a/b\leq t$, with $p$ not dividing $b$.
\end{thm}
\begin{proof}   Choose a line bundle $\cG = \cO_X(G)$ such that $\fra \otimes \cG$ is globally generated by sections $x_1, \ldots, x_m \in H^0(X, \fra \otimes \cG)$.  By Lemma~\ref{lem:FiniteSumNonPrinc} there is a Cartier divisor $H$ and integers $e_1, \ldots, e_s \gg 0 $ such that the test ideal $\tau(X, \Delta, \fra^w)$ is equal to
\[
\sum_{i=1}^s  \Image\left(F_*^{e_i} \Sym^c(x_1, \ldots, x_m) \otimes_k \cG^{-w(p^{e_i}-1)} \otimes  \cO_X( (1-p^{e_i}) (K_X+\Delta) - H) \xrightarrow{T^{e_i}} \cO_X\right)
 \]
 which is globally generated if each summand is. Fix now $A$ a globally generated ample Cartier divisor.  We claim it suffices to find a Cartier divisor $L_1$, such that
 \[
 \begin{array}{rl}
 & \big( F_*^{e_i} \cO_X( (1-p^{e_i}) (K_X+\Delta + w G) - H) \big) \otimes \cO_X(L_1 + (d+1)A) \\
 = & F_*^{e_i}\cO_X( (p^{e_i} - 1) (L_1 - \Delta - K_X - w G) + L_1 - H + p^{e_i}(d+1)A)
 \end{array}
 \]
  is globally generated (the equality in the displayed equation follows from the projection formula).  Indeed, assuming this global generation choose $L = L_1 + (d+1) A$ with $d = \dim X$ and note that the image of a globally generated sheaf is still globally generated.  We will find a single $L_1$ that works for all $0 \leq w \leq t$.

Since $\Rscr(X, -\Delta - K_X)$ is finitely generated, we can use Lemma~\ref{lem:FiniteGenNonPrinc} to find a Cartier divisor $M$ so that $(M - \Delta - K_X)$ is an ample Weil divisor. Moreover, we can find an ample Cartier divisor $N$ such that $N - tG$ is ample. Notice that, for all $0\leq w \leq t$,
\[
N - w G = \left(1-\frac{w}{t} \right) N + \frac{w}{t} (N - t G) \text{ is ample.}
 \]
 This observation is what lets us replace $tG$ with $wG$.
 Set $L_1 := M + N$. By \cite[Lem. 2.18(i)]{ChiecchioUrbinatiAmpleWeilDivisors}, $L_1 - K_X - \Delta - wG = (M - K_X - \Delta) + (N - wG)$ is an ample Weil divisor.

Fix $0 \leq w \leq t$. We now show that the \CM{} regularity, with respect to $A$, of
\[
F_*^e\cO_X( (p^e - 1) (L_1 - \Delta - K_X - wG) + L_1 - H+p^e(d+1)A)
 \]
 is zero for each $e = e_i \gg 0$, which guarantees by Mumford's theorem \cite[Thm. 1.8.3]{LazarsfeldPositivity1} the desired global generation. It suffices now to show that
\begin{equation}
\label{eq.DesiredVanishingForRegularity}
H^i( X, \cO_X(-i A) \otimes_{\cO_X} F^e_* \cO_X( (p^e - 1) ( L_1 - \Delta - K_X - wG) + L_1 - H + p^e (d+1) A) ) = 0
 \end{equation}
 which by the projection formula and the fact that $F^e_*$ doesn't change the underlying sheaves of Abelian groups, is the same as showing
 \[
 H^i ( X, \cO_X( (p^e - 1) (L_1 - \Delta - K_X - wG) + L_1 - H + p^e (d+1-i) A )) =0
 \]
for $0 < i \leq d$ and $d = \dim X$.
Since we may assume that $e \gg 0$, $ L_1 - H + p^e (d+1-i) A$ is nef.  Therefore because $L_1 - \Delta - K_X - wG$ is ample Weil and $L  - H + p^e (d+1-i) A$ is nef and Cartier, we may apply the version of Fujita vanishing \cite[Thm. 4.1]{ChiecchioUrbinatiAmpleWeilDivisors} to obtain the vanishing desired in \autoref{eq.DesiredVanishingForRegularity}.  This completes the proof.
\end{proof}

\begin{remark}
\label{rem.EffectiveDivisorForVanishing}
Indeed, it is not hard to choose $L$ effectively.  Summarizing the proof above, fix an ample Cartier $G$ so that $\mathfrak{a} \otimes \O_X(G)$ is globally generated, fix $A$ to be a globally generated ample Cartier divisor, and fix $M$ Cartier so that $M -\Delta-K_X$ is ample and choose an ample Cartier $N$ so that $N-tG$ is ample.  Then we can take
\[
L = (d+a)A + M + N.
\]
\end{remark}
We now turn to the promised results on discreteness and rationality.

\begin{prop}
\label{prop.DiscretenessOfFJumpingNumbersViaGlobal}
Suppose now that $X$ is normal and projective and $\Rscr(X,-K_X-\Delta)$ is finitely generated.  Then for any ideal sheaf $\fra$ on $X$, the jumping numbers of $\tau(X, \Delta, \fra^t)$ are without limit points.
\end{prop}
\begin{proof}
First assume that $\Delta$ has Weil-index not divisible by $p$.  It follows from an appropriately generalized version of the argument of \cite[Lemma 3.23]{BlickleSchwedeTakagiZhang} that $\tau(X, \Delta, \fra^t) = \tau(X, \Delta, \fra^{t+\varepsilon})$ for all $0 < \varepsilon \ll 1$.  Hence for every real number $t \in R_{\geq 0}$, there is a rational number $w = a/b$ with $p$ not dividing $p$ with $\tau(X, \Delta, \fra^t) = \tau(X, \Delta, \fra^w)$.   Now fix $t_0 > 0$.  It follows from Theorem \ref{thm.GGofNonPrincipalMult} that there exists a Cartier divisor $L$ such that $\tau(X, \Delta, \fra^w) \otimes \cO_X(L)$ is globally generated for every $w < t_0$ with $w = a/b$ and where $p$ does not divide $b$.  But then by our previous discussion, we also see that $\tau(X, \Delta, \fra^t) \otimes \cO_X(L)$ is globally generated for every $t < t_0$.  The discreteness follows since now for $0 \leq t < t_0$, $H^0(X, \tau(X, \Delta, \fra^t) \otimes \cO_X(L)) \subseteq H^0(X, \cO_X(L))$ form a decreasing chain of subspaces of a finite dimensional vector space $H^0(X, \cO_X(L))$, and of course by the global generation hypothesis, if $H^0(X, \tau(X, \Delta, \fra^{t_2}) \otimes \cO_X(L)) = H^0(X, \tau(X, \Delta, \fra^{t_1}) \otimes \cO_X(L))$ then $\tau(X, \Delta, \fra^{t_2}) = \tau(X, \Delta, \fra^{t_1})$.  This proves the result when $\Delta$ has Weil-index not divisible by $p$.

Next assume that $p^d \Delta$ has Weil-index divisible by $p$.  Fix a map $F^d_* \O_X(K_X) \to \O_X(K_X)$ inducing a map on the fraction fields $T : F^d_* \cK(X) \to \cK(X)$.  As in \cite[Theorem 6.25]{SchwedeTuckerTestIdealFiniteMaps}, this map induces a possibly non-effective Weil divisor $\mathcal{R}_{T} \sim (1-p^d) K_X$ with
\begin{equation}
\label{eq.TOfTauIsTau}
T \big(\tau(X, p^d \Delta - \mathcal{R}_T, (\fra^{[p^d]})^t) \big) = \tau(X, \Delta, \fra^t)
\end{equation}
Choose a Cartier divisor $G$ so that $p^d \Delta - \mathcal{R}_T + p^d G$ is effective and notice that it also has Weil-index not divisible by $p$.  Next observe that
\[
\begin{array}{rl}
& -K_X - p^d \Delta + \mathcal{R}_T  - p^d G \\
\sim & -K_X -p^d \Delta + (1-p^d)K_X - p^d G \\
= & -p^d (K_X + \Delta) - p^d G \\
= & F^{d*}(-K_X - \Delta - G)
\end{array}
\]
and hence that $\Rscr(X, -K_X-p^d \Delta + \mathcal{R}_T - p^d G)$ is finitely generated (note that the $-G$ is Cartier and thus harmless so we are really taking the $p^d$th Veronese of $\Rscr(X,-K_X-\Delta)$.  Hence by what we have already shown, the $F$-jumping numbers of $\tau(X, p^d \Delta - \mathcal{R}_T, (\fra^{[p^d]})^t)$ have no limit points.  Therefore by applying $T$ via \eqref{eq.TOfTauIsTau}, we see that the $F$-jumping numbers of $\tau(X, \Delta+G, \fra^t)$ also have no limit points.  But then by \cite[3.26]{BlickleSchwedeTakagiZhang} the $F$-jumping numbers of $\tau(X, \Delta, \fra^t)$ have no limit points proving the theorem.
\end{proof}

\subsection{Global generation and stabilization of $\sigma$}



We now give another proof of \autoref{cor.HartshorneSpeiserStabilization} in the projective setting.

\begin{thm}
\label{thm.GlobalizedSigmaStabilization}
Suppose that $(X, \Delta)$ is a projective pair such that $\Rscr(X,-K_X-\Delta)$ is finitely generated and $K_X + \Delta$ has Weil-index not divisible by $p$.  Then the images
\[
\Image\big( F^e_* \cO_X( (1-p^e)(K_X + \Delta))  \to \cO_X \big)
\]
stabilize for $e$ sufficiently large and divisible.  We use $\sigma(X, \Delta)$ to denote this stable image.
\end{thm}
\begin{proof}
Choose a globally generated ample Cartier divisor $A$ and a Cartier divisor $L$ such that $L - K_X - \Delta$ is an ample Weil divisor by \ref{lem:FiniteGenNonPrinc}.  For each $e$ such that $(p^e - 1)(K_X + \Delta)$ is integral, set
\[
\sigma_e(X, \Delta) = \Image\big( F^e_* \cO_X( (1-p^e)(K_X + \Delta))  \to \cO_X \big).
\]
Then fixing $d = \dim X$
\[
\sigma_e(X, \Delta) \otimes \cO_X(dA + L) = \Image\left( F^e_* \cO_X( p^e dA + L + (p^e-1)(L - K_X - \Delta))  \to \cO_X \right).
\]
We immediately notice that $F^e_* \cO_X( p^e dA + L + (p^e-1)(L - K_X - \Delta))$ is $0$-regular with respect to $A$ and hence its image $\sigma_e(X, \Delta) \otimes \cO_X(dA + L)$ is globally generated.  As the global generating sections all lie in $H^0(X, \cO_X(dA + L))$ which is finite dimensional, and as the $\sigma_e$ form a descending chain of ideals as $e$ increases, we see that $\sigma_e$ stabilizes for $e$ sufficiently large and divisible as claimed.
\end{proof}

As an immediate corollary of the proof, we obtain:

\begin{cor}
Suppose again that $(X, \Delta)$ is a projective pair of dimension $d$ such that $\Rscr(X,-K_X-\Delta)$ is finitely generated and that $K_X + \Delta$ has Weil-index not divisible by $p$.
If $L$ is a Cartier divisor such that $L - K_X - \Delta$ is an ample Weil divisor and if $A$ is a globally generated ample Cartier divisor, then $\sigma(X, \Delta) \otimes \cO_X(dA + L)$ is globally generated.
\end{cor} 

%% file: Alterations.tex
In this section we give a description of the test ideal $\tau(X, \Delta, \ba^t)$ under the assumption that $\sR(-K_X-\Delta)$ is finitely generated.  This generalizes \cite{BlickleSchwedeTuckerTestAlterations} from the case that $-K_X-\Delta$ is $\bQ$-Cartier.  As a consequence we obtain a generalization of a result of A.~K.~Singh's \cite{SinghSplintersFRegFG} \footnote{This result was announced years ago, but has not been distributed.}, also compare with \cite{SinghQGorensteinSplinters}.

Before starting in on this, let us fix notation and recall the following from \autoref{sec.Preliminaries}.
\begin{setting}
\label{set.SectionAlterations}
Suppose that $\Delta \geq 0$ is a $\bQ$-divisor on an $F$-finite normal scheme $X$, $\sR(-K_X-\Delta)$ is finitely generated with $X' = \Proj \sR(-K_X-\Delta) \xrightarrow{\mu} X$.  Suppose that $\fra$ is an ideal sheaf on $X$ and $t \geq 0$ is a real number.
\end{setting}

We have already seen that we can pullback $-K_X-\Delta$ to $X'$ by $\mu$ where it becomes a $\bQ$-Cartier divisor, see \autoref{lem.PullBackDivisorsOnSymbolicRees}.  Suppose further that $\pi : Y \to X$ is any alteration that factors through $X'$ as
\[
\xymatrix{
Y \ar@/_2pc/[rr]_{\pi} \ar[r]^{\psi} & X' \ar[r]^{\mu} & X. \\
& & \\
}
\]
Then we define $\pi^* (K_X +\Delta) = \psi^* \mu^* (K_X + \Delta)$, or equivalently we define $\pi^*$ is as in \cite{DeFernexHaconSingsOnNormal} (even though $\pi$ is not birational), see \autoref{subsect.pullbackWeildivisors}.
Recall, of course, that if $\pi : Y \to X$ is a small alteration (meaning that the non-finite-to-1 locus of $\pi$ has codimension $\geq 2$ in $Y$), then this coincides with the obvious pullback operation.  More generally, if $\pi : Y \to X$ is any alteration and $\rho : Y' \xrightarrow{\xi} Y \xrightarrow{\pi} X$ factors through both $\pi$ and $X' \xrightarrow{\mu} X$, and $\xi$ is birational, then we define $\pi^* (K_X + \Delta)$ to be $\xi_* \rho^*(-K_X-\Delta)$.

In the next lemma and later in the section, we use the notion and notation of parameter test modules $\tau(\omega_X, K_X + \Delta, \ba^t) := \tau(X, \Delta, \ba^t)$.  For a concise introduction and more about their relation to test ideals, please see \cite[Section 4]{SchwedeTuckerNonPrincipalIdeals}.

\begin{lemma}
\label{lem.TestIdealVsTestModuleForFGCanonical}
Working in \autoref{set.SectionAlterations}, if $m \in \bZ_{> 0}$ is such that $t m \in \bZ$, that $m \Delta$ is integral, and such that the $m$-th Veronese of the symbolic Rees algebra $\sR(X, -K_X-\Delta)$ is generated in degree $1$,
then
\[
\tau(X, \Delta, \ba^t) = \tau(\omega_X, (\O_X(-m(K_X + \Delta)) \ba^{tm})^{1 \over m}) = \tau(X, -K_X, (\O_X(-m(K_X + \Delta)) \ba^{tm})^{1 \over m}).
\]
\end{lemma}

\begin{proof}
We know by \autoref{lem.TestIdealConstruction} that for any sufficiently large Cartier $D \geq 0$ and any $e_0 \geq 0$ that
\[
\tau(X, \Delta, \ba^t) = \sum_{e \geq e_0} \Tr^e_X\Big(F^e_* \fra^{\lceil t p^e \rceil} \O_X( \lceil K_X - p^e(K_X + \Delta) - D \rceil)\Big).
\]
Choose $0 < D' \leq D''$ divisors such that $D''$ is Cartier and $K_X - D'$ is Cartier.  Since
\[
\begin{array}{rl}
& \tau(X, \Delta, \ba^t)\\
 = & \sum_{e \geq e_0} \Tr^e_X\Big(F^e_* \fra^{\lceil t p^e \rceil} \cdot \O_X( \lceil K_X - p^e(K_X + \Delta) - D - D'' \rceil)\Big)\\
 \subseteq & \sum_{e \geq e_0} \Tr^e_X\Big(F^e_* \fra^{\lceil t p^e \rceil} \cdot \O_X( \lceil K_X - D' - p^e(K_X + \Delta) - D  \rceil)\Big)\\
 = & \sum_{e \geq e_0} \Tr^e_X\Big(F^e_* \fra^{\lceil t p^e \rceil} \cdot \O_X( K_X - D') \cdot \O_X(\lceil - p^e(K_X + \Delta) - D \rceil)\Big)\\
   \subseteq & \sum_{e \geq e_0} \Tr^e_X\Big(F^e_* \fra^{\lceil t p^e \rceil} \cdot \O_X( K_X) \cdot \O_X(\lceil - p^e(K_X + \Delta) - D  \rceil)\Big)\\
 \subseteq & \sum_{e \geq e_0} \Tr^e_X\Big(F^e_* \fra^{\lceil t p^e \rceil} \cdot \O_X( \lceil K_X  - p^e(K_X + \Delta) - D  \rceil)\Big)\\
 = & \tau(X, \Delta, \ba^t)
\end{array}
\]
we see that $\tau(X, \Delta, \fra^t) = \sum_{e \geq e_0} \Tr^e_X\Big(F^e_* \fra^{\lceil t p^e \rceil} \cdot \O_X( K_X) \cdot \O_X(\lceil - p^e(K_X + \Delta) - D  \rceil)\Big)$.

This is already very close.
\begin{claim} We can choose a Cartier $D_3 > 0$ so that
\[
\O_X(-D_3) \cdot \fra^{\lceil tp^e \rceil} \cdot \O_X(\lceil -p^e(K_X + \Delta) \rceil) \subseteq \big( (\fra^{tm}) \cdot \O_X(-m(K_X + \Delta))\big)^{\lceil {1 \over m} \cdot p^e\rceil}
\]
for all $e \geq 0$.
\end{claim}
\begin{proof}[Proof of claim]
Checking this assertion is easy, we can certainly knock $\fra^{\lceil tp^e \rceil}$ into $(\fra^{tm})^{\lceil {1 \over m} p^e \rceil}$ by multiplication by a Cartier divisor.  Handling the other multiplication is a little trickier.  Likewise certainly we can multiply $\O_X(\lceil -p^e(K_X + \Delta) \rceil)$ into $\O_X( -\lceil {p^e \over m} \rceil m (K_X + \Delta) )$.  But then notice that $\O_X(-a m(K_X + \Delta)) = \O_X(-m(K_X + \Delta))^a$ by our finite generation hypothesis.  This proves the claim.
\end{proof}
Returning to the proof, we see that
\[
\begin{array}{rl}
& \tau(X, \Delta, \fra^t)\\
 =&  \sum_{e \geq e_0} \Tr^e_X\Big(F^e_* \fra^{\lceil t p^e \rceil} \cdot \O_X( K_X) \cdot \O_X(\lceil - p^e(K_X + \Delta) - D  - D_3 \rceil)\Big)\\
 \subseteq & \sum_{e \geq e_0} \Tr^e_X\Big(F^e_* (\O_X( K_X) \cdot (\fra^{tm})^{\lceil {1 \over m} p^e \rceil} \cdot \O_X( - m(K_X + \Delta))^{\lceil {1 \over m} p^e \rceil} \cdot \O_X( - D))\Big)\\
 = & \sum_{e \geq e_0} \Tr^e_X\Big(F^e_* (\O_X( K_X - p^e(K_X - K_X)) \cdot (\fra^{tm})^{\lceil {1 \over m} p^e \rceil} \cdot \O_X( - m(K_X + \Delta))^{\lceil {1 \over m} p^e \rceil} \cdot \O_X( - D))\Big)\\
 = & \tau(X, -K_X, (\O_X(-m(K_X + \Delta)) \ba^{tm})^{1 \over m})\\
 = & \tau(\omega_X, (\O_X(-m(K_X + \Delta)) \ba^{tm})^{1 \over m}) \\
 \subseteq & \sum_{e \geq e_0} \Tr^e_X\Big(F^e_* (\fra^{tm})^{\lceil {1 \over m} p^e \rceil} \cdot \O_X(\lceil K_X - p^e(K_X + \Delta) - D\rceil)\Big)\\
 = & \tau(X, \Delta, \fra^t)
\end{array}
\]
which proves the Lemma.
\end{proof}

\begin{remark}
It is tempting try to use \autoref{lem.TestIdealVsTestModuleForFGCanonical} to give another proof of discreteness and rationality of $F$-jumping numbers by appealing to \cite{SchwedeTuckerNonPrincipalIdeals}.  However, this doesn't seem to work.  In particular in \cite{SchwedeTuckerNonPrincipalIdeals} the authors did not prove discreteness and rationality of $F$-jumping numbers for $\tau(X, \Delta, \frb^s \fra^t)$ (as $t$-varies) -- mixed test ideals were not handled.  One could probably easily recover discreteness of $F$-jumping numbers via the usual arguments of gauge boundedness for Cartier algebras \cite{BlickleTestIdealsViaAlgebras} at least in the case when $X$ is finite type over a field.  For additional reading on mixed test ideals (and their pathologies) we invite the reader to look at \cite{PerezConstancyRegionsForMixedTestIdeals}.
\end{remark}

The really convenient thing about \autoref{lem.TestIdealVsTestModuleForFGCanonical} for our purposes is the following.

\begin{lemma}
\label{lem.BlowupPullbackVsDivisorPullback}
Using the notation of \autoref{lem.TestIdealVsTestModuleForFGCanonical},
suppose that $\pi : Y \to X$ is an alteration from a normal $Y$ where if we write $\frb = \O_X(-m(K_X + \Delta))$ then $\frb \cdot \O_Y = \O_Y(-T_Y)$ is an invertible sheaf.  Then $T_Y = m \pi^* (K_X + \Delta)$ where $\pi^* (K_X + \Delta)$ is defined as in the text below \autoref{set.SectionAlterations}.
\end{lemma}
\begin{proof}
This is easy. Indeed, we already know that $\pi$ factors through the normalized blowup of $\frb$ by the universal property of blowups.  On the other hand $\O_Y(-m\pi^*(K_X + \Delta)) = \frb \cdot \O_Y$.
\end{proof}

As a result, we immediately obtain the following.

\begin{theorem}
\label{thm.MainResultOnAlterations}
Suppose that $X$ is a normal $F$-finite integral scheme and that $\Delta$ on $X$ is an effective $\bQ$-divisor such that $S= \sR(-K_X-\Delta)$ is finitely generated.  Suppose also that $\fra$ is an ideal sheaf and $t \geq 0$ is a \emph{rational} number.  Then there exists an alteration $\pi : Y \to X$ from a normal $Y$, factoring through $X' = \sheafproj S$ and with $G = \Div_Y(\fra)$, so that
\begin{equation}
\label{eqn.AlterationImageIsTau}
\tau(X, \Delta, \fra^t) = \Image\big( \pi_* \O_Y(\lceil K_Y - \pi^*(K_X + \Delta) -tG\rceil) \to \O_X\big).
\end{equation}
This $\pi$ may be taken independently of $t \geq 0$ if desired.
If $\fra$ is locally principal (for instance if $\fra = \O_X$), then one may take $\pi$ to be a small alteration if desired.  Alternately, if $X$ is essentially of finite type over a perfect field, then one may take $Y$ regular by \cite{deJongAlterations}.

As a consequence we obtain that
\[
\tau(X, \Delta, \fra^t) = \bigcap_{\pi : Y \to X} \Image\big( \pi_* \O_Y(\lceil K_Y - \pi^*(K_X + \Delta + tG) \rceil) \to \O_X\big)
\]
where $\pi$ runs over all alterations with $\fra \cdot \O_Y = \O_Y(-G)$ is invertible (or all such regular alterations if $X$ is of finite type over a perfect field).
\end{theorem}
\begin{proof}
Most of the result follows immediately from \cite[Theorem A]{SchwedeTuckerNonPrincipalIdeals} combined with \autoref{lem.TestIdealVsTestModuleForFGCanonical} and \autoref{lem.BlowupPullbackVsDivisorPullback}.  Indeed, simply choose $m$ such that $tm$ is an integer and the condition of \autoref{lem.TestIdealVsTestModuleForFGCanonical} is satisfied, then apply \cite[Theorem A]{SchwedeTuckerNonPrincipalIdeals} to find an alteration such that the image of the above map is $\tau(X, -K_X, (\O_X(-m(K_X + \Delta)) \ba^{tm})^{1 \over m})$.  Consider the alterations $\pi$ that occur in the intersection $\bigcap_{\pi : Y \to X}$.  Following \cite[Theorem A]{SchwedeTuckerNonPrincipalIdeals} might seem to require that we only consider $\pi$ that factor through the normalized blowup of $\fra \cdot \O_X(-m(K_X + \Delta))$ (for $m$-sufficiently divisible).  However, it is easy to see that other $Y$'s can be dominated by those that factor through this blowup and the further blowups certainly have smaller images.

One also must handle the case of varying $t$ which is not quite done in \cite[Theorem A]{SchwedeTuckerNonPrincipalIdeals} in our generality (there the authors treated $\tau(X, \Delta, \fra^t)$, while here we need $\tau(X, -K_X, \frb^s \fra^t)$).  However, the argument there essentially goes through verbatim (alternately, this is the same argument as in \cite{SchwedeTuckerZhang}).

The only remaining part of the statement that doesn't follow immediately is the assertion in the case when $\fra$ is locally principal.  However, in the proof of \cite[Theorem A]{SchwedeTuckerNonPrincipalIdeals}, the alteration needed can always be taken to be a finite cover of the normalized blowup of the ideal (in this case, the normalized blowup of $\O_X(-m(K_X + \Delta)) \ba^{tm}$ which coincides with the normalized blowup of $\O_X(-m(K_X + \Delta))$).  This normalized blowup is of course $X'$ in our setting.
\end{proof}

In the above proof, our constructed $Y$ was definitely not finite over $X$.  This is different from \cite{BlickleSchwedeTuckerTestAlterations} where the simplest constructed $Y$ definitely was finite over $X$.  Fortunately, we can reduce to the case of a finite $Y$ at least when $\fra = \O_X$.

\begin{corollary}
\label{cor.TauFiniteAlterations}
Suppose that $X$ is a normal $F$-finite integral scheme and that $\Delta$ on $X$ is an effective $\bQ$-divisor such that $S= \sR(-K_X-\Delta)$ is finitely generated.  Then there exists a finite map $\phi : Y \to X$ from a normal $Y$, factoring through $X' = \sheafproj S$ such that
\[
\tau(X, \Delta) = \Image\big( \phi_* \O_Y(\lceil K_Y - \phi^*(K_X + \Delta) \rceil) \to \O_X\big).
\]
\end{corollary}
\begin{proof}
Let $\pi : Y' \to X$ be a small alteration satisfying \autoref{eqn.AlterationImageIsTau} from \autoref{thm.MainResultOnAlterations}.  Additionally assume that $\pi^* (K_X + \Delta)$ is integral (for simplicity of notation).  Next let $Y' \xrightarrow{\alpha} Y \xrightarrow{\phi} X$ be the Stein factorization of $\pi$.  Since $\alpha : Y' \to Y$ is small, we see that
\[
\pi_* \O_{Y'}(K_{Y'} - \pi^*(K_X + \Delta)) = \phi_* \O_Y(K_Y - \phi^*(K_X + \Delta))
\]
and the result follows.
\end{proof}

\begin{question}
Can one limit oneself to separable alterations in \autoref{thm.MainResultOnAlterations}?  In particular, is there always a \emph{separable} alteration $\pi : Y \to X$ with $$\tau(X, \Delta) = \Image\big( \pi_* \O_Y(\lceil K_Y - \pi^*(K_X + \Delta) \rceil) \to \O_X\big)?$$  The analogous result is known if $K_X + \Delta$ is $\bQ$-Cartier by \cite{BlickleSchwedeTuckerTestAlterations}.  However in our proof, $\pi$ is definitely not separable because we rely on \cite[Theorem A]{SchwedeTuckerNonPrincipalIdeals}, which uses Frobenius to induce certain vanishing results.  It is possible that this could be replaced by cohomology killing arguments as in for instance  \cite{BlickleSchwedeTuckerTestAlterations,BhattDerivedDirectSummand,HunekeLyubeznikAbsoluteIntegralClosure,HochsterHunekeInfiniteIntegralExtensionsAndBigCM}.
\end{question}

As a special case, we recover a result of Anurag K. Singh (that was announced years ago) \cite{SinghSplintersFRegFG}.
\begin{corollary}[Singh]
Suppose that $X$ is an $F$-finite splinter and $\sR(-K_X)$ is finitely generated.  Then $X$ is strongly $F$-regular.
\end{corollary}
\begin{proof}
Indeed, if $X$ is a splinter then for any finite morphism $\phi : Y \to X$, the evaluation-at-1 map $\sHom_{\O_X}(\phi_* \O_Y, \O_X) \to \O_X$ surjects.  However, $\sHom_{\O_X}(\phi_* \O_Y, \O_X) \cong \phi_* \O_Y(K_Y - \phi^* K_X)$ and the trace map to $\O_X$ is identified with the evaluation-at-1 map.  Hence using \autoref{cor.TauFiniteAlterations} we see that $\tau(X) = \tau(X, 0) = \O_X$.  Since for us $\tau(X)$ always denotes the big test ideal, this proves that $X$ is strongly $F$-regular.
\end{proof}


It would be natural to try to use the above to show that splinters are strongly $F$-regular for 3-dimensional varieties of characteristic $p > 5$, using the fact that such KLT varieties satisfy finite generation of their anticanonical rings \autoref{thm.FGforKLT}.  The gap is the following:

\begin{question}
Suppose that $R$ is a normal $F$-finite domain that is also a splinter.  Does there exist a $\bQ$-divisor $\Delta \geq 0$ on $\Spec R$ such that $K_X + \Delta$ is $\bQ$-Cartier and that $(\Spec R, \Delta)$ is KLT?
\end{question}

The analogous result on the existence of $\Delta$ for strongly $F$-regular varieties was shown in \cite{SchwedeSmithLogFanoVsGloballyFRegular}.  Of course, the fact that splinters are in fact derived splinters in characteristic $p > 0$ \cite{BhattDerivedDirectSummand} would likely be useful.

In particular, we do obtain the following.

\begin{corollary}
Suppose that $R$ is an $F$-finite three dimensional splinter which is also KLT (for an appropriate $\Delta \geq 0$) and that $R$ is finite type over an algebraically closed field of characteristic $p > 5$.  Then $R$ is strongly $F$-regular. 
\end{corollary}

%% file: Reduction.tex
The goal of this section is to show that multiplier ideals $\sJ(X, \Delta, \fra^t)$ reduce to test ideals $\tau(X_p, \Delta_p, \fra_p^t)$ after reduction to characteristic $p \gg 0$ at least if $\sR(-K_X-\Delta)$ is finitely generated.  We begin with some preliminaries on the reduction process.

Let $X$ be a scheme of finite type over an algebraically closed field $k$ of characteristic zero, $\Delta$ a $\QQ$-divisor, and $\fa \subset \O_X$ an ideal sheaf. One may choose a subring $A \subset k$ which is finitely generated over $\ZZ$ over which $X$, $\Delta$, and $\fa$ are all defined. Denote by $X_A$, $\Delta_A$, and $\fa_A \subset \O_{X_A}$ the models of $X$, $\Delta$, and $\fa$ over $A$. For any closed point $s \in \Spec A$ we denote the corresponding reductions $X_s$, $\Delta_s$, and $\fa_s \subset \O_{X_s}$ defined over the residue field $k(s)$ which is necessarily finite.  In the simple case where $A = \ZZ$, if $X_A = \Spec \ZZ[x_1,\ldots,x_n]/I$ for $I = (f_1,\ldots,f_m)$ and $p \in \ZZ$ is prime, the scheme $X_p = \Spec \FF_p[x_1,\ldots,x_n]/(f_1 \bmod p, \ldots, f_m \bmod p)$.

\begin{warning}
In what follows, we abuse terminology in the following way.  By $p \gg 0$, we actually mean the set of closed points of an open dense set $U \subseteq \Spec A$.  Furthermore, if we start with $X$ as above, by $X_p$ for $p \gg 0$ we actually mean some $X_s$ for some closed point $s$ in the aforementioned $U$.  This is a common abuse of notation and we do not expect it will cause any confusion.  It does substantially shorten statements of theorems.
\end{warning}

\begin{lem}\label{lem:sectionreduction}

Suppose $X$ is a normal quasi-projective variety over an algebraically closed field $k$ of characteristic zero. For any $\QQ$-divisor $\Gamma$ so that $\Rscr(X, \Gamma)$ is finitely generated we have $\Rscr(X, \Gamma)_p \cong \Rscr(X_p, \Gamma_p)$ for $p \gg 0$.

 In particular, if $\Delta$ is a $\QQ$-divisor and $\Rscr(X, -K_X-\Delta)$ is finitely generated, setting $X' = \Proj \Rscr(X, -K_X-\Delta)$ and $\mu \colon X' \to X$ we have $$\Rscr(X, -K_X-\Delta)_p = \Rscr(X_p, (-K_{X}-\Delta)_p) = \Rscr(X_p, -K_{X_p}-\Delta_p)$$ and so this ring is also finitely generated.  This means $(X')_p = (X_p)'$. We denote both by $X_p'$.
\end{lem}
\begin{proof}
Note that $\Rscr(X,\Gamma)_p$ makes sense for $p \gg 0$ as $ \Rscr(X,\Gamma)$ is finitely generated and  $\Rscr(X,\Gamma)_p$ is naturally finitely generated by the reduction of the generators of $\Rscr(X,\Gamma)$. The problem is that potentially the algebra $\Rscr(X,\Gamma)_p$ may not be the symbolic Rees algebra (local section algebra) $\bigoplus_{n \geq 0} \O_{X_p}(n \Gamma_p)$.   Throughout this proof we will constantly need to choose $p \gg 0$ (or technically, restrict to a smaller open subset $U$ of $\Spec A$).
First we record a claim that is certainly well known to experts.
\begin{claim}
For any Weil divisor $D$ and prime $p \gg 0$, potentially depending on $D$, we claim that $\O_X(D)_p \cong \O_{X_p} (D_p)$ as sheaves of $\O_{X_p}$-modules.
\end{claim}
\begin{proof}[Proof of claim]
To see this we prove that $\O_X(D)_p$ is reflexive and agrees outside a codimension $2$ subset with $\O_{X_p} (D_p)$.
Of course, since $\O_X(D)$ is reflexive, $$\O_X(D) \to \sHom_X(\sHom_X(\O_X(D), \O_X), \O_X)$$ is an isomorphism.  But this isomorphism is certainly preserved via reduction to characteristic $p$ so $(\O_X(D))_p$ is reflexive at least for $p \gg 0$.
Choose a closed set $Z \subseteq X$ of codimension 2, defined with no additional coefficients (other than the ones already needed to define $D$ and $X$) so that $D|_{X \setminus Z}$ is Cartier.  Note that $D_p|_{X_p \setminus Z_p}$ is Cartier and so $\O_{X_p \setminus Z_p}(D_p)$ is locally free and agrees with $(\O_{X \setminus Z}(D) )_p$, the claim follows.
\end{proof}
We return to the proof of the lemma.
Next define $X'$ to be the blowup of $\O_X(m\Gamma)$ for some $m \gg 0$ sufficiently divisible.  Then since $\mu : X' \to X$ is small, so is $\mu_p : X'_p \to X_p$, and note that $X'_p$ is still the blowup of $(\O_X(m\Gamma))_p = \O_{X_p}(m\Gamma_p)$ by the claim.  Since $m\Gamma' = m \pi^{-1}_* \Gamma$ was Cartier in characteristic zero, $m \Gamma_p'$ is Cartier after reduction to characteristic $p \gg 0$ as well.  Now $X'_p \to X_p$ is still small and notice that $\Gamma_p'$ is relatively ample (since $X'_p$ was obtained by blowing up $\O_{X_p}(m\Gamma_p)$).  Hence
\[
\bigoplus_{n \geq 0} (\mu_p)_* \O_{X'_p}(n \Gamma'_p) = \bigoplus_{n \geq 0} \O_{X_p}(n \Gamma_p)
\]
is finitely generated and has $\sheafproj$ equal to $X'_p$.  The lemma follows immediately.\end{proof}

Armed with this lemma, the proof of the main theorem for this section is easy.

\begin{theorem}
\label{thm.Reduction}
Suppose that $X$ is a normal quasi-projective variety over an algebraically closed field of characteristic zero.  Further suppose that $\Delta \geq 0$ is a $\bQ$-divisor such that $\sR(-K_X-\Delta)$ is finitely generated and also suppose that $\fra \subseteq \O_X$ is an ideal and $t \geq 0$ is a rational number.   Then
\[
\mJ(X, \Delta, \fra^t)_p = \tau(X_p, \Delta_p, \fra_p^t)
\]
for $p \gg 0$.
\end{theorem}
\begin{proof}
We know that $\mJ(X, \Delta, \fra^t) = \pi_* \O_{\tld X}(\lceil K_{\tld X} - {1 \over m} \pi^{\natural}m(K_X + \Delta) - tG\rceil)$ for some sufficiently divisible $m$ and sufficiently large log resolution of singularities, by \autoref{def.MultiplierIdealDefinition} (here we need that $\O_X(-m(K_X + \Delta)) \cdot \O_{\tld X} = \O_{\tld X}(-A)$ and $\fra \cdot \O_{\tld X} = \O_{\tld X}(-B)$ are invertible).  We rewrite this multiplier ideal as \[
\pi_* \O_{\tld X}(\lceil K_{\tld X} - {1 \over m} \pi^{\natural}m(K_X + \Delta) - tG\rceil) = \pi_* \O_{\tld X}(\lceil K_{\tld X} - \pi^{\natural}(K_X - K_X) - {1 \over m} A - tB\rceil)
 \]
and observe it is equal to $\mJ(X, -K_X, \O_X(-m(K_X+\Delta))^{1\over m} \cdot \fra^t)$.  Note that since $\sR(-K_X-\Delta)$ is finitely generated, the choice of $\tld X$ is independent of the choice of $m$ (at least for $m$ sufficiently divisible).

Since $K_X - K_X$ is Cartier, we know that
$$\mJ(X, -K_X, \O_X(-m(K_X+\Delta))^{1\over m} \cdot \fra^t)_p = \tau(X_p, -K_{X_p}, \O_{X_p}(-m(K_{X_p} + \Delta_p))^{1 \over m})$$
for $p \gg 0$ by \cite{TakagiInterpretationOfMultiplierIdeals,HaraYoshidaGeneralizationOfTightClosure}.  But \autoref{lem.TestIdealVsTestModuleForFGCanonical} shows that $\tau(X_p, -K_{X_p}, \O_{X_p}(-m(K_{X_p} + \Delta_p))^{1 \over m}) = \tau(X_p, \Delta_p, \fra_p^t)$.  Combining these equalities proves the result.
\end{proof}

\begin{remark}
Theorem~\ref{thm.Reduction} also implies that if $X_p$ is strongly $F$-regular for all $p \gg 0$ and $\sR(-K_X)$ is finitely generated, then $X$ is KLT.  Nobuo Hara gave a talk about this result at the conference in honor of Mel Hochster's 65th Birthday in 2008 but the result was never published.
\end{remark}

\begin{corollary}
Suppose $X$ is a variety over an algebraically closed field of characteristic zero that is KLT in the sense of \cite{DeFernexHaconSingsOnNormal}.  Then for any $\bQ$-divisor $\Delta \geq 0$ and any ideal sheaf $\fra$ and rational $t \geq 0$, we have that
\[
\mJ(X, \Delta, \fra^t)_p = \tau(X_p, \Delta_p, \fra_p^t)
\]
for $p \gg 0$.
\end{corollary}
\begin{proof}
It follows from the minimal model program \cite{BirkarCasciniHaconMcKernan} and in particular, \cite[Theorem 92]{KollarExercisesInBiratGeom} that $\sR(-K_X-\Delta)$ is finitely generated, the result follows immediately from \autoref{thm.Reduction}.
\end{proof}

%% file: master.bbl
\def\cfudot#1{\ifmmode\setbox7\hbox{$\accent"5E#1$}\else
  \setbox7\hbox{\accent"5E#1}\penalty 10000\relax\fi\raise 1\ht7
  \hbox{\raise.1ex\hbox to 1\wd7{\hss.\hss}}\penalty 10000 \hskip-1\wd7\penalty
  10000\box7}
\providecommand{\bysame}{\leavevmode\hbox to3em{\hrulefill}\thinspace}
\providecommand{\MR}{\relax\ifhmode\unskip\space\fi MR}
\providecommand{\MRhref}[2]{%
  \href{http://www.ams.org/mathscinet-getitem?mr=#1}{#2}
}
\providecommand{\href}[2]{#2}
\begin{thebibliography}{dFDTT15}

\bibitem[AM99]{AberbachMacCrimmonSomeResultsOnTestElements}
{\sc I.~M. Aberbach and B.~MacCrimmon}: \emph{Some results on test elements},
  Proc. Edinburgh Math. Soc. (2) \textbf{42} (1999), no.~3, 541--549.
  {\sf\scriptsize MR1721770 (2000i:13005)}

\bibitem[Bha12]{BhattDerivedDirectSummand}
{\sc B.~Bhatt}: \emph{Derived splinters in positive characteristic}, Compos.
  Math. \textbf{148} (2012), no.~6, 1757--1786. {\sf\scriptsize 2999303}

\bibitem[Bir13]{BirkarExistenceOfFlipsMinimalModels}
{\sc C.~Birkar}: \emph{Existence of flips and minimal models for 3-folds in
  char $p$}, arXiv:1311.3098, to appear in Annales Scientifiques de l'ENS.

\bibitem[BCHM10]{BirkarCasciniHaconMcKernan}
{\sc C.~Birkar, P.~Cascini, C.~D. Hacon, and J.~McKernan}: \emph{Existence of
  minimal models for varieties of log general type}, J. Amer. Math. Soc.
  \textbf{23} (2010), no.~2, 405--468. {\sf\scriptsize 2601039 (2011f:14023)}

\bibitem[BW14]{BirkarWaldronMoriFiberSpaces}
{\sc C.~Birkar and J.~Waldron}: \emph{Existence of {M}ori fibre spaces for
  3-folds in char $p$}, arXiv:1410.4511.

\bibitem[Bli13]{BlickleTestIdealsViaAlgebras}
{\sc M.~Blickle}: \emph{Test ideals via algebras of {$p^{-e}$}-linear maps}, J.
  Algebraic Geom. \textbf{22} (2013), no.~1, 49--83. {\sf\scriptsize 2993047}

\bibitem[BMS08]{BlickleMustataSmithDiscretenessAndRationalityOfFThresholds}
{\sc M.~Blickle, M.~Musta{\c{t}}{\v{a}}, and K.~E. Smith}: \emph{Discreteness
  and rationality of {$F$}-thresholds}, Michigan Math. J. \textbf{57} (2008),
  43--61, Special volume in honor of Melvin Hochster. {\sf\scriptsize 2492440
  (2010c:13003)}

\bibitem[BMS09]{BlickleMustataSmithFThresholdsOfHypersurfaces}
{\sc M.~Blickle, M.~Musta{\c{t}}{\u{a}}, and K.~E. Smith}:
  \emph{{$F$}-thresholds of hypersurfaces}, Trans. Amer. Math. Soc.
  \textbf{361} (2009), no.~12, 6549--6565. {\sf\scriptsize 2538604
  (2011a:13006)}

\bibitem[BS13]{BlickleSchwedeSurveyPMinusE}
{\sc M.~Blickle and K.~Schwede}: \emph{{$p^{-1}$}-linear maps in algebra and
  geometry}, Commutative algebra, Springer, New York, 2013, pp.~123--205.
  {\sf\scriptsize 3051373}

\bibitem[BSTZ10]{BlickleSchwedeTakagiZhang}
{\sc M.~Blickle, K.~Schwede, S.~Takagi, and W.~Zhang}: \emph{Discreteness and
  rationality of {$F$}-jumping numbers on singular varieties}, Math. Ann.
  \textbf{347} (2010), no.~4, 917--949. {\sf\scriptsize 2658149}

\bibitem[BST15]{BlickleSchwedeTuckerTestAlterations}
{\sc M.~Blickle, K.~Schwede, and K.~Tucker}: \emph{{$F$}-singularities via
  alterations}, Amer. J. Math. \textbf{137} (2015), no.~1, 61--109.
  {\sf\scriptsize 3318087}

\bibitem[BdFFU15]{BoucksomdeFernexFavreUrbinati}
{\sc S.~Boucksom, T.~de~Fernex, C.~Favre, and S.~Urbinati}: \emph{Valuation
  spaces and multiplier ideals on singular varieties}, 29--51. {\sf\scriptsize
  3380442}

\bibitem[CTX13]{CasciniTanakaXuBasePointFree}
{\sc P.~Cascini, H.~Tanaka, and C.~Xu}: \emph{On base point freeness in
  positive characteristic}, 1305.3502.

\bibitem[Chi14]{ChiecchioMMPnoflips}
{\sc A.~Chiecchio}: \emph{About a minimal model program without flips},
  arXiv:1410.6230.

\bibitem[CU15]{ChiecchioUrbinatiAmpleWeilDivisors}
{\sc A.~Chiecchio and S.~Urbinati}: \emph{Ample {W}eil divisors}, J. Algebra
  \textbf{437} (2015), 202--221. {\sf\scriptsize 3351963}

\bibitem[Cut88]{CutkoskyWeilDivisorsAndSymbolic}
{\sc S.~Cutkosky}: \emph{Weil divisors and symbolic algebras}, Duke Math. J.
  \textbf{57} (1988), no.~1, 175--183. {\sf\scriptsize 952230 (89e:14003)}

\bibitem[dFDTT15]{deFernexDocampoTakagiTuckerNumQGor}
{\sc T.~de~Fernex, R.~Docampo, S.~Takagi, and K.~Tucker}: \emph{Comparing
  multiplier ideals to test ideals on numerically {$\Bbb{Q}$}-{G}orenstein
  varieties}, Bull. Lond. Math. Soc. \textbf{47} (2015), no.~2, 359--369.
  {\sf\scriptsize 3335128}

\bibitem[DH09]{DeFernexHaconSingsOnNormal}
{\sc T.~{De Fernex} and C.~Hacon}: \emph{Singularities on normal varieties},
  Compos. Math. \textbf{145} (2009), no.~2, 393--414.

\bibitem[dJ96]{deJongAlterations}
{\sc A.~J. de~Jong}: \emph{Smoothness, semi-stability and alterations}, Inst.
  Hautes \'Etudes Sci. Publ. Math. (1996), no.~83, 51--93. {\sf\scriptsize
  1423020 (98e:14011)}

\bibitem[Dem88]{DemazureNormalGradedRings}
{\sc M.~Demazure}: \emph{Anneaux gradu\'es normaux}, Introduction \`a la
  th\'eorie des singularit\'es, {II}, Travaux en Cours, vol.~37, Hermann,
  Paris, 1988, pp.~35--68. {\sf\scriptsize MR1074589 (91k:14004)}

\bibitem[ELSV04]{EinLazSmithVarJumpingCoeffs}
{\sc L.~Ein, R.~Lazarsfeld, K.~E. Smith, and D.~Varolin}: \emph{Jumping
  coefficients of multiplier ideals}, Duke Math. J. \textbf{123} (2004), no.~3,
  469--506. {\sf\scriptsize MR2068967 (2005k:14004)}

\bibitem[FST11]{FujinoSchwedeTakagiSupplements}
{\sc O.~Fujino, K.~Schwede, and S.~Takagi}: \emph{Supplements to non-lc ideal
  sheaves}, Higher Dimensional Algebraic Geometry, RIMS K\^oky\^uroku Bessatsu,
  B24, Res. Inst. Math. Sci. (RIMS), Kyoto, 2011, pp.~1--47.

\bibitem[Gab04]{Gabber.tStruc}
{\sc O.~Gabber}: \emph{Notes on some {$t$}-structures}, Geometric aspects of
  Dwork theory. Vol. I, II, Walter de Gruyter GmbH \& Co. KG, Berlin, 2004,
  pp.~711--734.

\bibitem[GHNV90]{GotoHerrmannNishidaVillamayor}
{\sc S.~Goto, M.~Herrmann, K.~Nishida, and O.~Villamayor}: \emph{On the
  structure of {N}oetherian symbolic {R}ees algebras}, Manuscripta Math.
  \textbf{67} (1990), no.~2, 197--225. {\sf\scriptsize 1042238 (91a:13006)}

\bibitem[HX15]{HaconXuThreeDimensionalMinimalModel}
{\sc C.~D. Hacon and C.~Xu}: \emph{On the three dimensional minimal model
  program in positive characteristic}, J. Amer. Math. Soc. \textbf{28} (2015),
  no.~3, 711--744. {\sf\scriptsize 3327534}

\bibitem[Har01]{HaraInterpretation}
{\sc N.~Hara}: \emph{Geometric interpretation of tight closure and test
  ideals}, Trans. Amer. Math. Soc. \textbf{353} (2001), no.~5, 1885--1906
  (electronic). {\sf\scriptsize MR1813597 (2001m:13009)}

\bibitem[HY03]{HaraYoshidaGeneralizationOfTightClosure}
{\sc N.~Hara and K.-I. Yoshida}: \emph{A generalization of tight closure and
  multiplier ideals}, Trans. Amer. Math. Soc. \textbf{355} (2003), no.~8,
  3143--3174 (electronic). {\sf\scriptsize MR1974679 (2004i:13003)}

\bibitem[Har77]{Hartshorne}
{\sc R.~Hartshorne}: \emph{Algebraic geometry}, Springer-Verlag, New York,
  1977, Graduate Texts in Mathematics, No. 52. {\sf\scriptsize MR0463157 (57
  \#3116)}

\bibitem[Har94]{HartshorneGeneralizedDivisorsOnGorensteinSchemes}
{\sc R.~Hartshorne}: \emph{Generalized divisors on {G}orenstein schemes},
  Proceedings of Conference on Algebraic Geometry and Ring Theory in honor of
  Michael Artin, Part III (Antwerp, 1992), vol.~8, 1994, pp.~287--339.
  {\sf\scriptsize MR1291023 (95k:14008)}

\bibitem[HS77]{HartshorneSpeiserLocalCohomologyInCharacteristicP}
{\sc R.~Hartshorne and R.~Speiser}: \emph{Local cohomological dimension in
  characteristic {$p$}}, Ann. of Math. (2) \textbf{105} (1977), no.~1, 45--79.
  {\sf\scriptsize MR0441962 (56 \#353)}

\bibitem[Hoc07]{HochsterFoundations}
{\sc M.~Hochster}: \emph{Foundations of tight closure theory}, lecture notes
  from a course taught on the University of Michigan Fall 2007 (2007).

\bibitem[HH90]{HochsterHunekeTC1}
{\sc M.~Hochster and C.~Huneke}: \emph{Tight closure, invariant theory, and the
  {B}rian\c con-{S}koda theorem}, J. Amer. Math. Soc. \textbf{3} (1990), no.~1,
  31--116. {\sf\scriptsize MR1017784 (91g:13010)}

\bibitem[HH92]{HochsterHunekeInfiniteIntegralExtensionsAndBigCM}
{\sc M.~Hochster and C.~Huneke}: \emph{Infinite integral extensions and big
  {C}ohen-{M}acaulay algebras}, Ann. of Math. (2) \textbf{135} (1992), no.~1,
  53--89. {\sf\scriptsize 1147957 (92m:13023)}

\bibitem[HH94]{HochsterHunekeFRegularityTestElementsBaseChange}
{\sc M.~Hochster and C.~Huneke}: \emph{{$F$}-regularity, test elements, and
  smooth base change}, Trans. Amer. Math. Soc. \textbf{346} (1994), no.~1,
  1--62. {\sf\scriptsize MR1273534 (95d:13007)}

\bibitem[HL07]{HunekeLyubeznikAbsoluteIntegralClosure}
{\sc C.~Huneke and G.~Lyubeznik}: \emph{Absolute integral closure in positive
  characteristic}, Adv. Math. \textbf{210} (2007), no.~2, 498--504.
  {\sf\scriptsize 2303230 (2008d:13005)}

\bibitem[KLZ09]{KatzmanLyubeznikZhangOnDiscretenessAndRationality}
{\sc M.~Katzman, G.~Lyubeznik, and W.~Zhang}: \emph{On the discreteness and
  rationality of {$F$}-jumping coefficients}, J. Algebra \textbf{322} (2009),
  no.~9, 3238--3247. {\sf\scriptsize 2567418 (2011c:13005)}

\bibitem[KSSZ14]{KatzmanSchwedeSinghZhang}
{\sc M.~Katzman, K.~Schwede, A.~K. Singh, and W.~Zhang}: \emph{Rings of
  {F}robenius operators}, Math. Proc. Cambridge Philos. Soc. \textbf{157}
  (2014), no.~1, 151--167. {\sf\scriptsize 3211813}

\bibitem[Kol08]{KollarExercisesInBiratGeom}
{\sc J.~Koll{\'a}r}: \emph{Exercises in the birational geometry of algebraic
  varieties}, arXiv:0809.2579.

\bibitem[KM98]{KollarMori}
{\sc J.~Koll{\'a}r and S.~Mori}: \emph{Birational geometry of algebraic
  varieties}, Cambridge Tracts in Mathematics, vol. 134, Cambridge University
  Press, Cambridge, 1998, With the collaboration of C. H. Clemens and A. Corti,
  Translated from the 1998 Japanese original. {\sf\scriptsize MR1658959
  (2000b:14018)}

\bibitem[Kun76]{KunzOnNoetherianRingsOfCharP}
{\sc E.~Kunz}: \emph{On {N}oetherian rings of characteristic {$p$}}, Amer. J.
  Math. \textbf{98} (1976), no.~4, 999--1013. {\sf\scriptsize MR0432625 (55
  \#5612)}

\bibitem[Laz04]{LazarsfeldPositivity1}
{\sc R.~Lazarsfeld}: \emph{Positivity in algebraic geometry. {I}}, Ergebnisse
  der Mathematik und ihrer Grenzgebiete. 3. Folge. A Series of Modern Surveys
  in Mathematics [Results in Mathematics and Related Areas. 3rd Series. A
  Series of Modern Surveys in Mathematics], vol.~48, Springer-Verlag, Berlin,
  2004, Classical setting: line bundles and linear series. {\sf\scriptsize
  MR2095471 (2005k:14001a)}

\bibitem[Lip69]{LipmanRationalSingularities}
{\sc J.~Lipman}: \emph{Rational singularities, with applications to algebraic
  surfaces and unique factorization}, Inst. Hautes \'Etudes Sci. Publ. Math.
  (1969), no.~36, 195--279. {\sf\scriptsize MR0276239 (43 \#1986)}

\bibitem[Lyu97]{LyubeznikFModulesApplicationsToLocalCohomology}
{\sc G.~Lyubeznik}: \emph{{$F$}-modules: applications to local cohomology and
  {$D$}-modules in characteristic {$p>0$}}, J. Reine Angew. Math. \textbf{491}
  (1997), 65--130. {\sf\scriptsize MR1476089 (99c:13005)}

\bibitem[LS99]{LyubeznikSmithStrongWeakFregularityEquivalentforGraded}
{\sc G.~Lyubeznik and K.~E. Smith}: \emph{Strong and weak {$F$}-regularity are
  equivalent for graded rings}, Amer. J. Math. \textbf{121} (1999), no.~6,
  1279--1290. {\sf\scriptsize MR1719806 (2000m:13006)}

\bibitem[LS01]{LyubeznikSmithCommutationOfTestIdealWithLocalization}
{\sc G.~Lyubeznik and K.~E. Smith}: \emph{On the commutation of the test ideal
  with localization and completion}, Trans. Amer. Math. Soc. \textbf{353}
  (2001), no.~8, 3149--3180 (electronic). {\sf\scriptsize MR1828602
  (2002f:13010)}

\bibitem[Mus13]{MustataNonNefLocusPositiveChar}
{\sc M.~Musta{\c{t}}{\u{a}}}: \emph{The non-nef locus in positive
  characteristic}, A celebration of algebraic geometry, Clay Math. Proc.,
  vol.~18, Amer. Math. Soc., Providence, RI, 2013, pp.~535--551.
  {\sf\scriptsize 3114955}

\bibitem[P{\'e}r13]{PerezConstancyRegionsForMixedTestIdeals}
{\sc F.~P{\'e}rez}: \emph{On the constancy regions for mixed test ideals}, J.
  Algebra \textbf{396} (2013), 82--97. {\sf\scriptsize 3108073}

\bibitem[Sch09]{SchwedeFAdjunction}
{\sc K.~Schwede}: \emph{{$F$}-adjunction}, Algebra Number Theory \textbf{3}
  (2009), no.~8, 907--950. {\sf\scriptsize 2587408 (2011b:14006)}

\bibitem[Sch11]{SchwedeTestIdealsInNonQGor}
{\sc K.~Schwede}: \emph{Test ideals in non-{$\mathbb{{Q}}$}-{G}orenstein
  rings}, Trans. Amer. Math. Soc. \textbf{363} (2011), no.~11, 5925--5941.

\bibitem[SS10]{SchwedeSmithLogFanoVsGloballyFRegular}
{\sc K.~Schwede and K.~E. Smith}: \emph{Globally {$F$}-regular and log {F}ano
  varieties}, Adv. Math. \textbf{224} (2010), no.~3, 863--894. {\sf\scriptsize
  2628797 (2011e:14076)}

\bibitem[ST14a]{SchwedeTuckerTestIdealFiniteMaps}
{\sc K.~Schwede and K.~Tucker}: \emph{On the behavior of test ideals under
  finite morphisms}, J. Algebraic Geom. \textbf{23} (2014), no.~3, 399--443.
  {\sf\scriptsize 3205587}

\bibitem[ST14b]{SchwedeTuckerNonPrincipalIdeals}
{\sc K.~Schwede and K.~Tucker}: \emph{Test ideals of non-principal ideals:
  {C}omputations, jumping numbers, alterations and division theorems}, J. Math.
  Pures Appl. (9) \textbf{102} (2014), no.~5, 891--929. {\sf\scriptsize
  3271293}

\bibitem[STZ12]{SchwedeTuckerZhang}
{\sc K.~Schwede, K.~Tucker, and W.~Zhang}: \emph{Test ideals via a single
  alteration and discreteness and rationality of $f$-jumping numbers}, Math.
  Res. Lett. \textbf{19} (2012), no.~01, 191--197.

\bibitem[Sin99]{SinghQGorensteinSplinters}
{\sc A.~K. Singh}: \emph{{$\bold Q$}-{G}orenstein splinter rings of
  characteristic {$p$} are {F}-regular}, Math. Proc. Cambridge Philos. Soc.
  \textbf{127} (1999), no.~2, 201--205. {\sf\scriptsize 1735920 (2000j:13006)}

\bibitem[Sin14]{SinghSplintersFRegFG}
{\sc A.~K. Singh}: \emph{Private communication}.

\bibitem[Smi00]{SmithMultiplierTestIdeals}
{\sc K.~E. Smith}: \emph{The multiplier ideal is a universal test ideal}, Comm.
  Algebra \textbf{28} (2000), no.~12, 5915--5929, Special issue in honor of
  Robin Hartshorne. {\sf\scriptsize MR1808611 (2002d:13008)}

\bibitem[Tak04]{TakagiInterpretationOfMultiplierIdeals}
{\sc S.~Takagi}: \emph{An interpretation of multiplier ideals via tight
  closure}, J. Algebraic Geom. \textbf{13} (2004), no.~2, 393--415.
  {\sf\scriptsize MR2047704 (2005c:13002)}

\bibitem[Wat81]{WatanabeRemarksOnDemazure}
{\sc K.~Watanabe}: \emph{Some remarks concerning {D}emazure's construction of
  normal graded rings}, Nagoya Math. J. \textbf{83} (1981), 203--211.
  {\sf\scriptsize 632654 (83g:13016)}

\bibitem[Xu15]{XuBasePointFreePositiveChar}
{\sc C.~Xu}: \emph{On the base-point-free theorem of 3-folds in positive
  characteristic}, J. Inst. Math. Jussieu \textbf{14} (2015), no.~3, 577--588.
  {\sf\scriptsize 3352529}

\end{thebibliography}
